\documentclass[10pt]{amsart}

\author{Ties Laarakker}
\address{}
\email{p.t.a.laarakker@uu.nl}
\title[The Kleiman-Piene conjecture and node polynomials]{The Kleiman-Piene conjecture and node polynomials for plane curves in $\PP^3$}

\usepackage[utf8]{inputenc}
\usepackage{amsmath}
\usepackage{amsfonts}
\usepackage{amssymb}
\usepackage{amsthm}
\usepackage{caption}
\usepackage{enumerate}
\usepackage{fixltx2e}
\usepackage{mathtools}
\usepackage{setspace}
\usepackage{mathrsfs}
\usepackage{tikz-cd}
\usepackage{url}
\usepackage{breakurl}

\renewcommand{\O}{\mathcal{O}}
\newcommand{\C}{\mathcal{C}}
\newcommand{\F}{\mathcal{F}}
\renewcommand{\S}{\mathcal{S}}
\renewcommand{\L}{\mathcal{L}}
\newcommand{\LL}{\mathbb{L}}

\newcommand{\U}{\mathcal{U}}

\newcommand{\Z}{\mathcal{Z}}

\newcommand{\D}{\mathbf{D}}

\newcommand{\CC}{\mathbb{C}}
\newcommand{\PP}{\mathbb{P}}
\newcommand{\ZZ}{\mathbb{Z}}

\newcommand{\SM}{\mathrm{SM}}
\newcommand{\Gr}{\mathrm{Gr}}

\newcommand{\amp}{{\bf AMP}}
\newcommand{\dima}{{\bf DIM\textsubscript{KP}}}
\newcommand{\dimb}{{\bf DIM}}

\newcommand{\vir}{\mathrm{vir}}

\DeclareMathOperator{\proj}{Proj}
\DeclareMathOperator{\sym}{Sym}
\DeclareMathOperator{\hilb}{Hilb}

\newtheorem{result}{Theorem}

\newtheorem{theorem}{Theorem}[section]

\newtheorem{conjecture}[theorem]{Conjecture}
\newtheorem{corollary}[theorem]{Corollary}
\newtheorem{definition}[theorem]{Definition}
\newtheorem{lemma}[theorem]{Lemma}
\newtheorem{proposition}[theorem]{Proposition}

\theoremstyle{remark}
\newtheorem{notation}[theorem]{Notation}
\newtheorem{example}[theorem]{Example}
\newtheorem{problem}{Problem}
\newtheorem{remark}[theorem]{Remark}

\begin{document}
\renewcommand{\O}{\mathcal{O}}
\renewcommand{\L}{\mathcal{L}}

\begin{abstract}
For a relative effective divisor $\C$ on a smooth projective family of surfaces $q:\S\rightarrow B$, we consider the locus in $B$ over which the fibres of $\C$ are $\delta$-nodal curves. We prove a conjecture by Kleiman and Piene on the univerality of an enumerating cycle on this locus. We propose a bivariant class $\gamma(\C)\in A^*(B)$ motivated by the BPS calculus of Pandharipande and Thomas, and show that it can be expressed universally as a polynomial in classes of the form $q_*(c_1(\O(\C))^a c_1(T_{\S/B})^b c_2(T_{\S/B})^c)$. Under an ampleness assumption, we show that $\gamma(\C)\cap[B]$ is the class of a natural effective cycle with support equal to the closure of the locus of  $\delta$-nodal curves. Finally, we will apply our method to calculate node polynomials for plane curves intersecting general lines in $\PP^3$. We verify our results using 19th century geometry of Schubert.
\end{abstract}

\maketitle
\section{Introduction}
\subsection{The Kleiman-Piene Conjecture}
All schemes we consider are separated and of finite type over $\CC$. Let $B$ be a base scheme, and let $p\colon\S\rightarrow B$ be a smooth family of surfaces, i.e.\ a smooth projective morphism of relative dimension $2$. By a \emph{curve}, we mean a proper $1$-dimensional scheme, not necessary irreducible or reduced. Let $\C$ be a \emph{relative effective (Cartier) divisor} on the family $\S\rightarrow B$, i.e.\ an effective Cartier divisor on $\S$, such that the morphism $\C\rightarrow B$ is flat. Fix a non-negative integer $\delta$. We call a curve \emph{$\delta$-nodal} if it is reduced, has $\delta$ nodes and no other singularities. Consider the following counting problem:

\begin{problem}
What is, if finite, the number of $\delta$-nodal curves in the family $\C\rightarrow B$?
\end{problem}\noindent
More generally, consider the locus
\[B(\delta)\coloneqq\big\{b\in B \mid \C_b \mbox{ is a }\delta\mbox{-nodal curve}\big\}\]
and write $\overline{B(\delta)}$ for its closure in $B$.

\begin{problem}
What is the class $\left[\overline{B(\delta)}\right]\in A_*(B)$?
\end{problem}

Assume that $B$ is Cohen-Macaulay and of pure dimension $n$. In \cite{KP2}, Kleiman and Piene construct a natural effective cycle $U(\delta)$ with support equal to the closure of the locus of $\delta$-nodal curves. For $\delta\leq 8$ they prove that the class $[U(\delta)]$ is given in a rather specific form as a polynomial in the classes
\[\epsilon(a,b,c)\coloneqq p_*(c_1(\O_\S(\C))^a c_1(T_{\S/B})^bc_2(T_{\S/B})^c)\,.\]

Kleiman and Piene work with certain assumptions on the dimensions of equisingular strata in the family. For a curve $C$, let $\D$ be its equisingularity type. It can be represented by an Enriques diagram, which encodes the numerical invariants of the singularities of $C$ \cite{KP1}. Conversely, for an equisingularity type (or Enriques diagram) $\D$, we write $B(\D)\subset B$ for the locus of curves in the family $\C\rightarrow B$, with equisingularity type $\D$.

One of the invariants of an equisingularity type is the codimension $cod(\D)$. It is the `expected codimension' in which curves with equisingularity type $\D$ appear in a family. More precisely, in \cite{KP1}, it is characterized as the codimension of the locus of curves with equisingularity type $\D$ in the universal family $\C\rightarrow |L|$ of any sufficiently ample complete linear system.
The hypotheses on the family $\C\rightarrow B$ under which the class $U(\delta)$ is constructed and which we will denote by \dima\ are the following:
\begin{itemize}
\item The locus of non-reduced curves $B(\infty)$ has codimension $>\delta$;
\item For each equisingularity type $\D$, the locus $B(\D)$ has at least the expected codimension $cod(\D)$, or codimension $>\delta$.
\end{itemize}\noindent
Here we use the convention $\mbox{codim}(\emptyset)=\infty$.
%
%
In \cite{KP2} and \cite{KP3}, the authors prove the following theorem:

\newcommand{\FootnoteBell}{
In fact, their statement is more precise: The polynomials are of the form $P_\delta(a_1,\dots,a_\delta)/\delta!$, in which $P_\delta$ is the $\delta$-th Bell polynomial, and $a_i$ is a linear combination of classes $\epsilon(a,b,c)$ with $a+b+2c= i + 2$, so that $a_i\in A^i(B)$. Moreover, an algorithm is given that produces these classes.
}

\begin{theorem}[Kleiman-Piene]\label{KP}
Under the above hypotheses \dima, the locus $B(\delta)$ of $\delta$ nodal curves is either empty, or has pure codimension $\delta$. There is a natural non-negative cycle $U(\delta)$ with support $\overline{B(\delta)}$. For $\delta\leq 8$, the rational equivalence class $[U(\delta)]$ is given by a universal polynomial\footnote{\FootnoteBell} in the classes $\epsilon(a,b,c)$.
\end{theorem}

Moreover, in \cite{KP2} the following conjecture is made.

\begin{conjecture}[Kleiman-Piene]\label{conj}
Theorem \ref{KP} holds for all $\delta\geq0$.
\end{conjecture}



In this paper we propose a class $\gamma(\C)\in A^\delta(B)$, enumerating the $\delta$-nodal curves, inspired by the BPS calculus of Pandharipande and Thomas \cite{PT}. We will show that if $B$ is complete, but not necessarily Cohen-Macaulay, the class $\gamma(\C)\cap [B]$ is the rational equivalence class of a natural cycle with support $\overline{B(\delta)}$. For this we work with hypotheses \dimb, similar to but slightly weaker than \dima, and an additional ampleness assumption \amp. By means of a family version of an algorithm by Ellingsrud, G\"ottsche and Lehn \cite{EGL}, we show that without assumptions, the class $\gamma(\C)$ is a universal polynomial in the classes $\epsilon(a,b,c)$.  This will be the content of Theorem \ref{Theorem A} below.


\subsection{BPS numbers}
Let $C$ be a locally planar curve of arithmetic genus $g$. In \cite{PT} the authors consider the following transformation of the generating series of topological Euler characteristics of Hilbert schemes $C^{[i]}$ of $i$ points on $C$, which defines the \emph{BPS numbers} $n_{r,C}$ of $C$.
\[\sum_{i=0}^\infty e(C^{[i]})\,q^i = \sum_{r=-\infty}^g n_{r,C}\,q^{g-r} (1-q)^{2r-2} \, .\]
They prove the following:
\begin{theorem}[Pandharipande-Thomas]\label{PT}
The numbers $n_{r,C}$ are zero, unless $g - \delta \leq r \leq g$, where $\delta$ is the $\delta$-invariant of $C$, i.e., $g-\delta$ is the geometric genus of $C$.
\end{theorem}

Shende proves in \cite{Shende2012} that the number $n_{g-i,C}$ equals the degree of the subvariety of $i$-nodal curves in the versal deformation space of $C$. In particular it is positive for $0\leq i \leq \delta$.

Let $B$ be a scheme and let $p\colon\C\rightarrow B$ be a family of curves, i.e., a projective flat morphism of relative dimension $1$. Assume that the fibres are locally planar curves of genus $g$. Let
\[p^{[i]}\colon \C_B^{[i]}= \hilb^i(\C/B)\rightarrow B\]
be the relative Hilbert scheme of $i$ points on the fibres of $\C\rightarrow B$. We define constructible functions $n_r=n_r(\C)$ on $B$ by
\begin{equation}\label{top}
\sum_{i=0}^\infty p_*^{[i]}(1_{\C_B^{[i]}})\,q^i = \sum_{r=-\infty}^g n_r\,q^{g-r} (1-q)^{2r-2} \, .
\end{equation}
In other words, $n_r$ is the function that assigns the number $n_{r,\C_b}$ to a point $b\in B$. By Theorem \ref{PT}, the function $n_{g-\delta}$ has support on curves with $\delta$-invariant $\geq \delta$. In the same paper it is shown that $n_{g-\delta,C}=1$ for a $\delta$-nodal curve $C$.

Let $\S\rightarrow B$ be a smooth family of surfaces and let $\C\subset \S$ be a relative effective divisor. We will use the embedding $\C\hookrightarrow\S$ to analogously define classes $n_r^{\vir}\in A^*B$. In fact, $\C_B^{[i]}$ is a subscheme of $\S_B^{[i]}=\hilb^i(\S/B)$, cut out transversely, i.e., in the expected codimension $i$, and regularly by the tautological bundle $\O(\C)_B^{[i]}$ (see Lemma \ref{lci}). The scheme $\S_B^{[i]}$ is smooth over $B$ \cite{AIK} and the \emph{virtual tangent bundle} $T_{\C_B^{[i]}/B}$, as defined in \cite[B.7.6]{Fu}, is given by the class
\[\left[\left.T_{\S_B^{[i]}/B}\right|_{\C_B^{[i]}}-\left.\O(\C)_B^{[i]}\right|_{\C_B^{[i]}}\right]\]
in the Grothendieck group $K(\C_B^{[i]})$ of vector bundles on $\C_B^{[i]}$. Let
\[c\colon K \Rightarrow (A^*)^\times\]
be the total Chern class. Then the classes $n^{\vir}_r = n_r^{\vir}(\C)\in A^*(B)$ are defined by the equation
\begin{equation}\label{vir}
\sum_{i=0}^\infty p_*^{[i]}c(T_{\C_B^{[i]}/B})\, q^i = \sum_{r=-\infty}^g n_r^{\vir}\,q^{g-r} (1-q)^{2r-2} \, .
\end{equation}
Here the homomorphism
\[p^{[i]}_*\colon A^*(\C_B^{[i]})\rightarrow A^*(B)\]
denotes the Gysin push-forward as defined in \cite[Chapter 17]{Fu}. We define
\[\gamma(\C) = \left\{n_{g-\delta}^{\vir}(\C)\right\}_\delta\in A^\delta(B)\]
to be equal to the degree-$\delta$ part of $n_{g-\delta}^{\vir}(\C)$. We will show that it reflects some of the properties of $n_{g-\delta}(\C)$. In fact, in Proposition \ref{defect}, we will compare $n_r$ and $n^\vir_r$ by means of the Chern-Schwartz-MacPherson class.

\begin{remark}
G\"ottsche and Shende \cite{GS14} also mention the Chern-Schwartz-MacPherson class of the constructible function $n_r$ as an invariant counting nodal curves. Moreover they consider an analogous class, using the virtual tangent bundle of Hilbert schemes of points of the curve. However, the use of the \emph{relative} tangent bundle, which is natural from the point of view of \cite{KP2}, is essential for our results.
\end{remark}


\subsection{Results}\label{results}
Recall that a line bundle $L$ on a smooth projective surface $S$ is called \emph{$\delta$-very ample} if for any finite subscheme $Z\subset S$ of length $\delta+1$, the map $H^0(S,L)\rightarrow H^0(Z,L|_Z)$ is surjective \cite{BS}. For a line bundle $\L$ on a smooth family of surfaces $\S\rightarrow B$, consider the following ampleness hypotheses, which we denote by \amp:

\vbox{
\begin{itemize}
\item For every $b\in B$, the line bundle $\L_b= \L|_{\S_b}$ on $\S_b$ is $\delta$-very ample.
\item The dimension of the vector spaces $H^0(\S_b,\L_b)$ is locally constant on $B$.
\end{itemize}}

Now let $\C$ and $\S\rightarrow B$ be given as above, and let $\L = \O(\C)$. Without making any assumptions on the dimensions of the equisingular strata, \amp\, guarantees that the class $\gamma(\C)\cap [B]$ is supported on the locus of curves with $\delta$-invariant $\geq \delta$. If $B$ is equidimensional, it will follow (Proposition \ref{sup}) that $\gamma(\C)\cap [B]$ is the class of a natural effective cycle with support $\overline{B(\delta)}$ if we assume the following hypotheses, which we denote by \dimb:
\begin{itemize}
\item The locus of $\delta$-nodal curves, if non-empty, has codimension $\delta$.
\item The loci of curves of the following type have codimension $>\delta$:
\begin{itemize}
\item Curves with $\delta$-invariant $>\delta$;
\item Curves with $\delta$-invariant $=\delta$, but with singularities other than nodes;
\item Non-reduced curves.
\end{itemize}
\end{itemize}\noindent
As explained in \cite{KP1}, for an equisingularity type $\D$ we have $cod(\D)\geq \delta(\D)$, with equality only for $\delta$-nodal curves. It follows that \dimb\, is slightly weaker than \dima. To summarize, we will prove the following theorem:

\begin{result}\label{Theorem A}
Let $B$ be a scheme and fix an integer $\delta$. Let $\C$ be a relative effective divisor on a flat family of smooth surfaces $\S \rightarrow B$. Then the class $\gamma(\C)$ can be expressed universally as a polynomial of degree $\delta$ in classes of the form $\epsilon(a,b,c)=p_*(c_1(\O(\C))^a c_1(T_{\S/B})^b c_2(T_{\S/B})^c)$.
Now assume that $B$ is complete of pure dimension $n$, that the line bundle $\O_\S(\C)$ satisfies \amp, and moreover assume that $\C\rightarrow B$ satisfies \dimb. Then the class $\gamma(\C)\cap[B]\in A_{n-\delta}(B)$ is the class of a natural cycle on $B$ with support $\overline{B(\delta)}$.  
\end{result}


The conjecture of Kleiman and Piene is a family version of the G\"ottsche conjecture \cite{Go}. For a sufficiently ample line bundle $L$ on a smooth surface $S$, the latter asserts that the degree of the Severi locus of $\delta$-nodal curves in the complete linear system $|L|$, is given by a universal polynomial in the classes $L^2$, $(L.K)$, $K^2$ and $c_2(S)$, for $K$ the canonical divisor on $S$. Equivalently, the number of $\delta$-nodal curves in a general linear system $\PP^\delta\subset |L|$ is given by such a polynomial.

The G\"ottsche conjecture was first proved using algebraic methods by Tzeng in \cite{Tz}. Other proofs were given in \cite{Liu}, \cite{Ka} and \cite{KST}. In \cite{LT} and \cite{Re} the result is generalized to other singularity types.

Our theorem implies the G\"ottsche conjecture for $\delta$-very ample $L$, but it is not independent from existing results. In fact, our method can be seen as a family version of the proof in \cite{KST}. Moreover, sharper results are known in terms of the required ampleness \cite{KS}.

\subsection{Application to plane curves in $\PP^3$}
The motivation for the project was the following counting problem. For fixed integers $\delta\geq0$ and $d>1$ write
\[n= \frac{d(d+3)}{2}+3-\delta\,.\]
and consider lines $\ell_1,\ldots,\ell_n\subset \PP^3$. The space of curves of degree $d$ that lie on a plane in $\PP^3$ and that intersect the lines $\ell_1,\ldots,\ell_n$ has expected dimension $\delta$. We will show that, if we choose the lines $\ell_1,\ldots,\ell_n$ sufficiently general, the subspace of $\delta$-nodal curves is finite (and reduced, as a scheme). For $d \geq \delta$, we can use our method to calculate the number $N_{\delta,d}$ of $\delta$-nodal plane curves of degree $d$ intersecting the lines $\ell_1,\ldots,\ell_n$.

Let $\C\rightarrow B$ be the universal plane curve of degree $d$ in $\PP^3$. We will show in Proposition \ref{finred} that for $d\geq\delta$, we have $N_{\delta,d} = \gamma(\C)\cap [B_{\ell_1,\dots,\ell_n}]$, in which $B_{\ell_1,\ldots,\ell_n}\subset B$ is the closed subvariety of curves intersecting the general lines $\ell_1,\ldots,\ell_n$. We use this to prove our second main result:

\begin{result}\label{Theorem B}
Let $\delta\geq 0$. The number of planar $\delta$-nodal curves of degree $d\geq \delta$ in $\PP^3$ intersecting $n=\frac{d(d+3)}{2}+3-\delta$ general lines is given by a polynomial $N_\delta(d)$ in $d$ of degree $\leq 9+2\delta$. Moreover, for $\delta\leq 12$ these polynomials are the ones given in Appendix \ref{Appendix A}.
\end{result}

\subsection*{Acknowledgements}
I thank Ritwik Mukherjee for useful discussions and for providing polynomials for the counting problem in $\PP^3$ that allowed us to verify our results, which was very helpful in an early stage of the project (see remark \ref{Ritwik}). I thank Ragni Piene, Jørgen Rennemo, and my supervisor Martijn Kool for useful discussion and comments on my work. In particular, I thank Martijn for suggesting this project.

\section{Preliminaries}
\subsection{Chern-Schwartz-MacPherson classes}
To any constructible function $f$ on a complete scheme $X$, one can assign a class $c_{SM}(f)$ in the Chow group of $X$, called the \emph{Chern-Schwartz-MacPherson class} of $f$. The existence of well-behaved Chern classes for constructible functions was conjectured by Deligne and Grothendieck and proved by MacPherson in \cite{Ma}. Several other constructions are known. See \cite{Al} for an overview and a new construction in a more general set-up.

For a subset $V$ of a scheme $X$, write $1_V\colon X\rightarrow \ZZ$ for the function with constant value $1$ on its support $V$. Recall that a constructible function is a map $f\colon X\rightarrow \ZZ$ that can be written as a finite sum \[f=\sum_{i\in I}\alpha_i 1_{V_i}\] with $\alpha_i\in\ZZ$ and $V_i\subset X$ closed. Let $F(X)$ be the group of constructible functions on $X$. For a proper morphism $g:X\rightarrow Y$, there is a homomorphism $g_*\colon F(X)\rightarrow F(Y)$ given by
\[g_*(1_V)(y)=e(V\cap g^{-1}(y)),\hspace{1em} y\in Y\, ,\]
in which $V\subset X$ is closed and $e$ is the topological Euler characteristic. For a scheme $X$ let $A_*(X)$ denote the Chow group of $X$. In the following theorem, we will view $A_*$ as a (covariant) functor on the category of complete schemes with proper morphisms to the category of abelian groups. Let $c$ denote the total Chern class.

\begin{theorem}[MacPherson]\label{MP}
There is a unique natural transformation $c_{SM}\colon F\Rightarrow A_*$ satisfying
$c_{SM}(1_X)=c(T_X)$ for $X$ smooth projective.
\end{theorem}
\begin{remark}
The uniqueness of such a natural transformation follows from resolution of singularities. MacPherson proved the naturality of the homology class, but in fact his argument gives this stronger result (see \cite[19.1.7]{Fu}).
\end{remark}

\begin{definition}
For a complete scheme $X$ and a constructible function $f\in F(X)$, we call $c_{SM}(f)=c_{SM}(X)(f)$ the \emph{Chern-Schwartz-MacPherson class} of $f$. We also write
\[c_{SM}(X)=c_{SM}(1_X)\]
and call this class the Chern-Schwartz-MacPherson class of $X$.
\end{definition}

It follows directly from the definitions that for a complete scheme $X$, we have
\[\int c_{SM}(X) = \pi_*(1_X) = e(X)\]
with $\pi:X\rightarrow \{*\}$ the morphims to a point. At the other extreme, we have the following lemma.

\begin{lemma}\label{smtop}
Let $X$ be a scheme and let $V\subset X$ be a locally closed subset of dimension $n$ and let $\overline{V}$ be its closure (with the reduced scheme structure) . Then
\[c_{SM}(1_V) = [\overline{V}] + \mbox{cycles of dimension}<n\,.\]
\end{lemma}

\begin{proof}
Let $g\colon\widetilde{V}\rightarrow \overline{V}$ be a birational morphism from a nonsingular projective variety $\widetilde{V}$. Let $U\subset \overline{V}$ be a dense open over which $g$ is an isomorphism and let $Z=\overline{V}\backslash U$ be its complement. Write
\[\partial V = \overline{V}\backslash V\]
for the boundary of $V$. Then the we have
\begin{align*}
g_*1_{\widetilde{V}}		&= 1_{\overline{V}}+ f\\
						&= 1_V + 1_{\partial V}+ f
\end{align*}
with $f$ a constructible function on $Z$. On the other hand, we have
\[c_{SM}(1_{\widetilde{V}})=c(T_{\widetilde{V}}) = [\widetilde{V}]+ \mbox{cycles of dimension}<n\,.\]
Since the functions $f$ and $1_{\partial V}$ are supported on closed subsets of dimension $<n$, by naturality the same holds for their Chern-Schwartz-Macpherson classes. It follows that we have
\begin{align*}
c_{SM}(1_V)	&= g_*c_{SM}(\widetilde{V})-c_{SM}(f)-c_{SM}(1_{\partial V})\\
			&= [\overline{V}] + \mbox{cycles of dimension}<n\,.\qedhere
\end{align*}
\end{proof}

\subsection{Hilbert schemes of points}\label{Hilb}
Let $S$ be a smooth projective surface, and let $S^{[n]}$ be the Hilbert scheme of $n$ points on $S$. Let $\Z$ be the universal subscheme of length $n$, with natural morphisms
\[
\begin{tikzcd}
\Z\arrow[r,"i"]\arrow[d,"\pi"]& S\\
S^{[n]} &
\end{tikzcd}
\]
For a vector bundle $F$ on $S$, we write
$F^{[n]}\coloneqq \pi_*i^*F$
for the tautological vector bundle on $S^{[n]}$ with fibre
\[\left.F^{[n]}\right|_{[Z]} = H^0(Z,F|_Z)\]
over a point $[Z]\in S^{[n]}$.

In \cite{EGL}, Ellingsrud, G\"ottsche and Lehn describe a method to calculate certain tautological integrals on $S^{[n]}$. In fact, they give a constructive proof of the following theorem:

\begin{theorem}[Ellingsrud, G\"ottsche, Lehn]\label{EGL}
Let $F_1,\ldots,F_l$ be vector bundles on $S$ of respective ranks $r_1,\ldots,r_l$. Let $P$ be a polynomial in the Chern classes of $T_{S^{[n]}}$ and the Chern classes of the bundles $F_i^{[n]}$. Then there is a universal polynomial $Q$, depending only on $P$, in numbers
\[\int_S p(T_S,F_1,\ldots,F_l)\,,\]
in which $p$ is a polynomial in the Chern classes of $T_S$, the ranks $r_i$ and the Chern classes of the bundles $F_i$, such that
\[\int_{S^{[n]}} P = Q\,.\]
\end{theorem}

Now let $B$ be a base scheme and let $q\colon\S\rightarrow B$ be proper and smooth of relative dimension $2$. Let $\S_B^{[i]}=\hilb^i(\S/B)$ denote the relative Hilbert scheme of $i$ points on the fibres of $\S\rightarrow B$, with structure morphism $q^{[i]}\colon\S_B^{[i]}\rightarrow B$. For a vector bundle $\F$ on $\S$ we can define as above the tautological bundle $\F_B^{[i]} = \pi_*i^*\F$ on $\S_B^{[i]}$, in which $\pi$ and $i$ are the natural morphisms in
\[
\begin{tikzcd}
\Z\arrow[r,"i"]\arrow[d,"\pi"]& \S\arrow[d, "q"]\\
\S_B^{[i]} \arrow[r,"q^{[i]}"] & B
\end{tikzcd}
\]
from the universal length $n$ subscheme $\Z$ on the fibres of $\S\rightarrow B$.
It restricts to
\[\left.\F_B^{[i]}\right|_{\S_b^{[i]}} = (\F|_{\S_b})^{[i]}\]
on the fibre $\S_b^{[i]}$ over a point $b\in B$. 

The following is proved in \cite{AIK} and generalizes a well known result by Fogarty:

\begin{lemma}\label{vlak}
The morphism $q^{[i]}\colon\S_B^{[i]}\rightarrow B$ is smooth of relative dimension $2i$.
\end{lemma}

Now let $\S\rightarrow B$ be given as above, and let $\C\subset \S$ be a relative effective divisor.

\begin{lemma}\label{lci}
The Hilbert scheme $\C_B^{[i]}$ is cut out transversely, i.e., in the expected codimension $i$, and regularly by the canonical section of the vector bundle $\O(\C)_B^{[i]}$ on $\S_B^{[i]}$. Moreover, the morphism $\C_B^{[i]}\rightarrow B$ is flat.
\end{lemma}
\begin{proof}
We follow \cite{AIK}. Let $b\in B$ be an arbitrary point and consider the fibres $C=\C_b$ and $S=\S_b$ over $b$. By \cite[§11.3.8]{EGA43}, it suffices to check that the Hilbert scheme $C^{[i]}$ is cut transversely and regularly by the canonical section of the bundle $\O(C)^{[i]}$ on $S^{[i]}$. Since $S^{[i]}$ is smooth, we only need to check that for a divisor $C\subset S$ on a smooth surface $S$, the Hilbert scheme $C^{[i]}$ has the expected dimension $i$. But this can be seen by inspection of the fibres of the Hilbert-Chow morphism $C^{[i]}\rightarrow C^{(i)}$. In fact, by \cite{Ia}, the locus in $S^{[i]}$ of subschemes of length $i$ supported at a point has dimension $i-1$. From this it follows directly that the locus in $C^{(i)}$ over which the fibres of the morphism $C^{[i]}\rightarrow C^{(i)}$ have dimension $r$, has codimension $\geq r$.
\end{proof}


\section{The smooth case}\label{relkst}
Let $B$ be a scheme. Let $\S\rightarrow B$ be smooth projective of relative dimension $2$ and let $\C\subset \S$ be a relative effective divisor. For a fixed $\delta$, let
\[n^\vir_{g-\delta}=n^\vir_{g-\delta}(\C)\in A^*(B)\]
be the class defined by equation \eqref{vir} in the introduction. The following situation is the model for our results.
\begin{proposition}\label{defect}
Assume that $B$ is \emph{projective} and that the relative Hilbert schemes of points $\C_B^{[i]}$ for $i=0,\ldots,\delta$ are non-singular. (So in particular $B$ is non-singular). Then we have the following identity in $A_*(B)$:
\begin{equation}\label{eqDefect}
c(T_B) \, n_{g-\delta}^{\vir}\cap[B] = c_{\SM}(n_{g-\delta})\, .
\end{equation}
In particular, the class $n_{g-\delta}^{\vir}\cap[B]$ is supported on the locus of curves with $\delta$-invariant $\geq\delta$, i.e. it is the push forward of a class on this locus. If $B$ is of pure dimension $n$ and the family of curves $\C\rightarrow B$ satisfies \dimb, we find
\[n_{g-\delta}^{\vir} \cap [B]= \left[\overline{B(\delta)}\right] + \beta\]
with $\beta$ a sum of cycles of dimension $<n-\delta$.
\end{proposition}
\begin{proof}
By the defining equation \eqref{top}, the constructible function $n_{g-\delta}$ is a linear combination of the terms $p_*^{[i]}(1_{\C_B^{[i]}})$. It is easy to see that only the terms with $i=0,\ldots,\delta$ are involved. Similarly, $n^{\vir}_{g-\delta}$ is a linear combination (with the same coefficients) of the classes $p^{[i]}_*(c(T_{\C_B^{[i]}/B}))$, with $i=0\ldots \delta$. Therefore it suffices to verify that relation \eqref{eqDefect} holds for the first $\delta +1$ terms in the left hand sides of \eqref{top} and \eqref{vir}.

Recall that we have defined the class $T_{\C_B^{[i]}/B}$ in the Grothendieck group $K(\C_B^{[i]})$ by
\[T_{\C_B^{[i]}/B} = \left.T_{\S_B^{[i]}/B} \right|_{\C_B^{[i]}} - \left.\O(\C)^{[i]}\right|_{\C_B^{[i]}}\,.\]
For $i = 0,\ldots,\delta$, we have by Lemma \ref{vlak} and by Lemma \ref{lci} the following relations in $K(\S_B^{[i]})$ and $K(\C_B^{[i]})$ respectively:
\begin{align*}
T_{\S_B^{[i]}}												& = T_{\S_B^{[i]}/B} + T_B\,;\\
\left.T_{\S_B^{[i]}}\right|_{\C_B^{[i]}}		& = T_{\C_B^{[i]}} + \left.\O(\C)_B^{[i]}\right|_{\C_B^{[i]}}\,.
\end{align*}
It follows that
\[T_{\C_B^{[i]}/B} = T_{\C_B^{[i]}} - T_B\]
and hence
\[c(T_{\C_B^{[i]}/B}) = c(T_{\C_B^{[i]}})\, c(T_B)^{-1}\in A^*(\C_B^{[i]})\,.\] By the defining properties of Chern-Schwartz-MacPherson classes (Theorem \ref{MP}) we obtain
\begin{align*}
c_{SM}(p_*^{[i]}(1_{\C_B^{[i]}}))	& = p_*^{[i]}c_{SM}(1_{\C_B^{[i]}})\\
													& = p_*^{[i]}{\left(c(T_{\C_B^{[i]}})\cap [\C_B^{[i]}]\right)}\\
													& = p_*^{[i]}{\left(c(T_B)\,c(T_{\C_B^{[i]}/B}) \cap[\C_B^{[i]}]\right)}\\
													& = c(T_B)\,p_*^{[i]}{\left(c(T_{\C_B^{[i]}/B})\right)}\cap [B]\, .
\end{align*}

By Theorem \ref{PT} the support of the constructible function $n_{g-\delta}$ is lies in locus in $B$ over which the curves in the family $\C\rightarrow B$ have $\delta$-invariant $\geq\delta$. By the functoriality of the Chern-Schwartz-MacPherson class, the cycle class
\[c(T_B) \, n_{g-\delta}^{\vir}\cap[B]\]
is the push-forward of a cycle class on this locus. Since $c(T_B)$ is invertible, the same holds for $n^{\vir}_{g-\delta}\cap[B]$. By \cite[Prop. 3.23]{PT}, the function $n_{g-\delta}$ has constant value $1$ on the locus of $\delta$-nodal curves. If $\C$ satisfies \dimb, it follows that the support of the constructible function $n_{g-\delta}-1_{B(\delta)}$ lies in a closed subset of codimension $>\delta$. Hence the last assertion follows from Lemma \ref{smtop}.
\end{proof}

\begin{example}\label{exkst}
In \cite{KST} it is shown that both conditions of the proposition are satisfied by the universal curve $\C\rightarrow |L|$ over the linear system of a $\delta$-very ample line bundle $L$ on a smooth surface $S$, i.e., $\C$ satisfies \dimb, and the relative Hilbert schemes $\C_{|L|}^{[i]}$ are non-singular for $i\leq \delta$. By Bertini's theorem it then follows that the same holds for the restriction $\C_{\PP^\delta}$ of the universal curve to a general linear system $\PP^\delta \subset |L|$. In particular the set $\PP^\delta(\delta)$ is finite, and it follows that the degree
\[\int_{\PP^\delta} n^\vir_{g-\delta}(\C_{\PP^\delta})\cap [\PP^\delta]= \int_{\PP^\delta} \left[\PP^\delta(\delta)\right]\]
equals the number of $\delta$-nodal curves in the linear system $\PP^\delta$.
\end{example}
\begin{remark}\label{reKST}
It should be noted that the integrals differ slightly from the ones in \cite{KST}. In fact, in loc.\ cit.\ the authors consider a linear combination of the integrals
\begin{equation}\label{eqKSTInt}
\int_{\C^{[i]}_{\PP^\delta}} c\left(T_{\C^{[i]}_{\PP^\delta}}\right)
\end{equation}
whereas we use the relative (virtual) tangent bundles
\[T_{\C^{[i]}_{\PP^\delta}/\PP^\delta}=\left.T_{\C^{[i]}_{|L|}/|L|}\right|_{\C^{[i]}_{\PP^\delta}}\]
and consider the integrals
\begin{equation}\label{eqOurInt}
\int_{\C^{[i]}_{\PP^\delta}} c\left(T_{\C^{[i]}_{\PP^\delta}/\PP^\delta}\right)
\end{equation}
Interestingly, \eqref{eqKSTInt} does not equal \eqref{eqOurInt} in general, but after taking the BPS linear combination of the integrals for $i=0,\ldots,\delta$, they both calculate the number of $\delta$ nodal curves in the linear system $\PP^\delta$.
\end{remark}

More generally, let $B$ be a scheme and let $p\colon\S\rightarrow B$ be smooth projective of relative dimension $2$. Let $\L$ be a line bundle on $\S$ and assume $\L$ satisfies \amp. Then $p_*\L$ is a vector bundle and the fiber of the projective bundle
\[|\L/B|\coloneqq \PP(p_*\L) = \proj(\sym((p_*\L)^*))\rightarrow B\]
over a point $b\in B$ is the $\delta$-very ample complete linear system $|\L_b|$. Consider the family
\[\S_{|\L/B|}=\S\times_B|\L/B|\rightarrow |\L/B|\]
and let $\C\subset \S_{|\L/B|}$ be the universal divisor.
\begin{lemma}\label{keyex}
The family of curves $\C\rightarrow |\L/B|$ satisfies \dimb. If $B$ is smooth, the Hilbert schemes $\C_{|\L/B|}^{[i]}$ are non-singular.
\end{lemma}
\begin{proof}
The first statement can be checked fibrewise over $B$.  By Lemma \ref{lci}, the Hilbert schemes of points $\C_{|\L/B|}^{[i]}$ are flat over $|\L/B|$, and hence over $B$.  Hence, if $B$ is smooth, also the second statement can be checked fibrewise over $B$. Hence, the lemma follows from Example \ref{exkst}.
\end{proof}


\section{Functoriality and support}
Let $B$ be a scheme. Let $\C$ be a relative effective divisor on a smooth family of surfaces $q\colon\S\rightarrow B$. Let $f\colon B'\rightarrow B$ be a morphism and consider the relative effective divisor $\C_{B'}=\C\times_BB'$ on the smooth family of surfaces
\[\S_{B'}=\S\times_BB'\rightarrow B'\,.\]
Then we have the following:
\begin{lemma}\label{func}
\begin{enumerate}[i)]
\item $f^*(n_{g-\delta}^{\vir}(\C))=n_{g-\delta}^{\vir}(\C_{B'})$ in $A^*(B')$;
\item $f_*(n_{g-\delta}^{\vir}(\C_{B'})\cap[B'])=n_{g-\delta}^{\vir}(\C)\cap f_*[B']$ in $A_*(B)$.
\end{enumerate}
\end{lemma}
\begin{proof}
It suffices to show the relations for the LHS of equation \eqref{vir}. Note that we have Cartesian squares
\[\begin{tikzcd}
\C_{B'}^{[i]}\arrow[r, "f_i"]\arrow[d,"p_{B'}^{[i]}"]	& \C_B^{[i]}\arrow[d, "p_B^{[i]}"]\\
B'\arrow[r, "f"] 												& B \,.
\end{tikzcd}\]
By flatness of the vertical maps, we have
\begin{align*}
f ^* p_{B*}^{[i]} c(T_{\C_B^{[i]}/B})		&=  p_{B'*}^{[i]} f_i^*c(T_{\C_B^{[i]}/B})\\
															&=  p_{B'*}^{[i]} c(f_i^*T_{\C_B^{[i]}/B})\\
															&=  p_{B'*}^{[i]} c(T_{\C_{B'}^{[i]}/B'})\,.
\end{align*}
The second part follows from the projection formula.
\end{proof}

Let $\C$ be given as above. Write $\L=\O(\C)$ and assume $\L$ satisfies \amp. As in the previous section, we can form the projective bundle $|\L/B| \coloneqq \PP(q_*\L)$ over $B$, with fibre over $b$ the $\delta$-very ample complete linear system $|\L_b|$. Let $\C''\subset \S_{|\L/B|}$ denote universal relative effective divisor on the family $\S_{|\L/B|}\rightarrow |\L/B|$. The bundle $\pi\colon |\L/B|\rightarrow B$ has a canonical section
\begin{equation}\label{eqSection}
\begin{tikzcd}
\vert\L/B\vert \arrow[d, "\pi", swap]	\\ B \arrow[u,bend right, "s", swap]
\end{tikzcd}
\end{equation}
induced by the inclusion $\O(-\C)\hookrightarrow \O$, and the divisor $\C$ is obtained by restricting the divisor $\C''$.

\begin{lemma}\label{lemdom}
Assume that $B$ is equidimensional, and let $W\subset \overline{|\L/B|(\delta)}$ be an irreducible component. Then $\pi(W)$ is an irreducible component of $B$.
\end{lemma}
\begin{proof}
Clearly $\pi(W)$ is closed and irreducible, so it suffices to show that $\dim(\pi(W)) = \dim(B)$. Write $W^\circ$ for the open $W\cap |\L/B|(\delta)$ of $W$. By Lemma \ref{keyex} and the first statement of Theorem \ref{KP}\footnote{Our assumptions are slightly weaker, but the argument given in \cite{KP2} still holds if we replace \dima by \dimb.}, $|\L/B|(\delta)$ has pure codimension $\delta$. Hence $W^\circ\subset |\L/B|$ has codimension $\delta$. By Example \ref{exkst}, a fibre $W^\circ_b$ over $b\in B$ is empty, or has codimension $\delta$ in $|\L/B|_b = |\L_b|$. It follows that
\begin{align*}
\dim(\pi(W)) 	&= \dim(\pi(W^\circ))\\
						&= \dim(|\L/B|) - \delta - (\dim(|\L_b|) - \delta)\\
						&= \dim(B)\,.\qedhere
\end{align*}
\end{proof}

\begin{definition}
Define a cycle
\[U(\delta) \coloneqq \sum_{W} l(\O_{\pi(W),B})\,[W] \in A_*(|\L/B|)\]
in which the sum is taken over the irreducible components of $\overline{|\L/B|(\delta)}$. Here $l(\O_{\pi(W),B})$ denotes the length of the local ring along the subvariety $\pi(W)$ of $B$.
\end{definition}

As in the introduction, we will use the notation
\[\gamma(\C) = \{n^\vir_{g-\delta}(\C)\}_\delta\in A^\delta(B)\,.\]

\begin{proposition}\label{sup}
Let $B$ be a \emph{complete} base scheme, and $\C$ be a relative effective divisor on a smooth family of surfaces $\S\rightarrow B$. Assume the line bundle $\L = \O(\C)$ satisfies \amp. Then we have
\[n_{g-\delta}^{\vir}(\C) \in A^{\geq\delta}(B)\, .\]
Now assume that $B$ is equidimensional, and let $s$ be given as in \eqref{eqSection}. Then we have
\[\gamma(\C)\cap[B] = s^! (U(\delta))\,.\]
In particular, $\gamma(\C)\cap[B]$ is supported on the locus of curves with $\delta$-invariant $\geq \delta$. Finally, if moreover $\C$ satisfies \dimb, then $\gamma(\C)\cap[B]$ is the class of a natural cycle with support $\overline{B(\delta)}$.
\end{proposition}
\begin{proof}
Assume that $B$ is equidimensional. Without loss of generality, we may assume that $B$ is connected so that $H^0(\S_b,\L_b)$ is constant and $|\L/B|$ equidimensional. By resolution of singularities, there is a proper surjective morphism $f\colon B'\rightarrow B$ from a smooth projective scheme $B'$ with $f_*[B'] = [B]$. In fact, we consider the union of $l(\O_{Z,B})$ copies of $Z$ for each irreducible component $Z\subset B$. Then we choose for each component a birational morphism from a smooth projective variety.

Let $\C'$, $\S'$ and $\L'$ be the base changes of $\C$, $\S$, and $\L$ along $f$. Then
\[\L'=f_\S^*\O_\S(\C)=\O_{\S'}(\C')\,\]
satisfies \amp\ and we have the following diagram with Cartesian squares:
\[
\begin{tikzcd}
\C'''\arrow[d]\arrow[r]														&\C''\arrow[d]\\
\vert\L'/B'\vert\arrow[d,"\pi'", swap]\arrow[r,"f_\pi"]		& \vert\L/B\vert\arrow[d, "\pi",swap]\\
B'\arrow[r,"f"]																	& B \arrow[u,bend right=40, "s", swap]\,,
\end{tikzcd}
\]
in which $\C'''$ and $\C''$ are the universal divisors on $\S'_{|\L'/B'|}$ and $\S_{|\L/B|}$ in the linear systems $|\L'/B'|$ and $|\L/B|$ respectively.

Note that $(f_\pi)_*[|\L'/B'|] = [|\L/B|]$. By Lemma \ref{keyex} and Proposition \ref{defect} we have
\[n_{g-\delta}^{\vir}(\C''')\cap |\L'/B'| = \left[\overline{|\L'/B'|(\delta)}\right]+ \alpha \in A_*(|\L'/B'|)\,,\]
with $\alpha$ a sum of classes with codimension $>\delta$.
%
We have (set-theoretically)
\[f_\pi(\overline{|\L'/B'|(\delta)}) = \overline{|\L/B|(\delta)}\,.\]
Now let $W$ be an irreducible component of $\overline{|\L/B|(\delta)}$. We will prove that the multiplicity of $W$ in $(f_\pi)_*([\overline{|\L'/B'|(\delta)}])$ is given by $l(\O_{\pi(W),B})$ ($\star$). By Lemma~\ref{lemdom}, $\pi(W)$ (with the reduced scheme structure) is an irreducible component of $B$. Any irreducible component $W'$ of $\overline{|\L'/B'|(\delta)}$ that maps onto $W$ lies over one of the $l(\O_{\pi(W),B})$ copies of a resolution of singularities
\[\rho\colon\widetilde{\pi(W)}\rightarrow \pi(W)\,,\]
so $W'$ lies in $W\times_{\pi(W)} \widetilde{\pi(W)}$. Let $U \subset \pi(W)$ be an open over which $\rho$ is an isomorphism. Then
\[W\times_{\pi(W)} \rho^{-1}(U)\rightarrow W\cap\pi^{-1}(U)\]
is an isomorphism. Since $W'$ is irreducible and maps onto $W$, we have
\[W' = \overline{W\times_{\pi(W)} \rho^{-1}(U)}\]
and the morphism $W'\rightarrow W$ is generically of degree $1$. We have now proved $\star$. It follows that
\[(f_\pi)_*\left(\left[\overline{|\L'/B'|(\delta)}\right]\right) = U(\delta)\,.\]

Hence, by Lemma \ref{func} we have
\begin{align*}
n_{g-\delta}^{\vir}(\C)\cap [B] 	&= s^*(n_{g-\delta}^{\vir}(\C''))\cap[B] \\
								&= s^!(n_{g-\delta}^{\vir}(\C'')\cap [|\L/B|])\\
								&= s^!(n_{g-\delta}^{\vir}(\C'')\cap (f_\pi)_*[|\L'/B'|])\\
								&= s^!(f_\pi)_*(n_{g-\delta}^{\vir}(\C''')\cap [|\L'/B'|])\\
								&= s^!(f_\pi)_*\left(\left[\overline{|\L'/B'|(\delta)}\right]\right)+\beta\\
								&= s^!\left(U(\delta)\right)+\beta
\end{align*}
with $\beta\in A_*(B)$ a sum of classes with codimension $>\delta$. In particular we find
\[\gamma(\C)\cap [B] = s^!\left(U(\delta)\right)\,.\]

If $\C$ satisfies \dimb, it follows that $s(B)$ and $\overline{|\L/B|(\delta)}$ intersect properly in $|\L/B|$, i.e.\ in dimension $\dim B - \delta$, and we have set theoretically
\begin{equation}\label{eqLocusDoorsnede}
s^{-1}\left(\overline{|\L/B|(\delta)}\right) = \overline{B(\delta)}\,.
\end{equation}
To see this, note that since $|\L/B|(\delta)$ has codimension $\delta$ in $|\L/B|$, an irreducible component of $s^{-1}\left(\overline{|\L/B|(\delta)}\right)$ has codimension $\leq \delta$. On the other hand, it consists of curves with $\delta$-invariant $\geq\delta$. So by \dimb, it has pure codimension $\delta$ in $B$. Moreover, the set
\[\left.s^{-1}\left(\overline{|\L/B|(\delta)}\right)\middle\backslash\right. B(\delta)\]
consists of curves with $\delta$-invariant $\geq \delta$ that are not $\delta$-nodal. By \dimb, it has codimension $>\delta$. Hence $B(\delta)$ lies dense in $s^{-1}\left(\overline{|\L/B|(\delta)}\right)$, proving equation \eqref{eqLocusDoorsnede}.

It follows that $\gamma(\C)\cap [B]$ is a the class of a natural effective cycle with support equal to $\overline{B(\delta)}$.

Now let $B$ be any complete scheme and let $V\rightarrow B$ be a morphism from an $n$-dimensional variety $V$. By the above we have
\[n^\vir_{g-\delta}(\C)\cap[V] = n^\vir_{g-\delta}(\C_V)\cap[V]\in A_{\leq n-\delta}(V)\]
and hence we have an equality of bivariant classes
\[n^\vir_{g-\delta}(\C) = \gamma(\C) + \alpha \in A^*(B)\]
with $\alpha\in A^*(B)$ a sum of classes of degree $>\delta$.
\end{proof}

\section{Universality: relative EGL}\label{SectionUni}
Let $B$ be a scheme, and let $\C$ be a relative effective divisor on a smooth family of surfaces $q\colon\S\rightarrow B$. The arithmetic genus of a curve in the family $p\colon\C\rightarrow B$ is denoted by $g$, which we view as a locally constant function on $B$. We consider the transformation of power series~\eqref{vir} and rewrite it as follows:
\begin{align*}
\sum_{n=0}^\infty p_*^{[n]}(c(T_{\C^{[n]}/B})\,q^n
			&= \sum_{r=-\infty}^g n^\vir_r\,q^{g-r} (1-q)^{2r-2}\\
			&= \sum_{i=0}^\infty n^\vir_{g-i}\,q^i (1-q)^{2(g-i)-2}\\
			&= \sum_{i=0}^\infty n^\vir_{g-i}\,q^i \sum_{j=0}^\infty (-q)^j\binom{2(g-i)-2}{j}\\
			&= \sum_{n=0}^\infty\sum_{i=0}^n n^\vir_{g-i}\ (-1)^{n-i}\binom{2(g-i)-2}{n-i}\,q^n\,.
\end{align*}
It follows that we can write
\begin{equation}\label{gooddef}
n_{g-\delta}^{\vir} = \sum_{i=0}^\delta a_i\,p^{[i]}_*(c(T_{\C_B^{[i]}/B}))\,,
\end{equation}
in which the $a_i$ are polynomials of degree $\delta - i$ in $g$, depending only on $\delta$ and $i$. In fact, $a_i=a_{i\delta}$ can be found by inverting the upper triangular matrix with $1$'s on the diagonal
\[\left[(-1)^{j-i}\binom{2(g-i)-2}{j-i}\right]_{0\leq i,j\leq \delta}\, {\textstyle \raisebox{-1em}{.}}\]

\begin{proposition}\label{propUniversal}
The class $\gamma(\C)$, can be expressed universally as a polynomial of degree $\delta$ in the classes
\[\epsilon(a,b,c)\coloneqq q_*(c_1(\O(\C))^a c_1(T_{\S/B})^b c_2(T_{\S/B})^c)\, ,\]
in which $q_*$ denotes the Gysin push-forward.
\end{proposition}

We first prove the following lemma:
\begin{lemma}\label{lemmaUniversal}
There exists a polynomial as in the proposition of degree $\leq\delta$ in the classes $\epsilon(a,b,c)$.
\end{lemma}
\begin{proof}
We can view $g$ as an element of $A^0(B)$. In fact we have
\begin{align*}
2g - 2 		&= p_*(c_1(T_{\C/B}^\vee))\\
			&= p_*(c_1(\O(\C)|_\C)-c_1(T_{\S/B}|_\C))\\
			&= q_*((c_1(\O(\C)))^2-c_1(\O(\C))c_1(T_{\S/B}))\,.
\end{align*}
By the equation \eqref{gooddef}, it suffices therefore to prove that the degree-$\delta$ parts of the classes $p^{[i]}_*(c(T_{\C_B^{[i]}/B}))$ can be expressed universally as polynomials of degree $i$ in the classes $\epsilon(a,b,c)$. By Lemma \ref{lci}, we have the equality
\begin{equation}\label{eqInt}
p^{[i]}_*(c(T_{\C_B^{[i]}/B}))=q^{[i]}_*\left(\frac{c(T_{\S_B^{[i]}/B})}{c(\O(\C)^{[i]})} c_i(\O(\C)^{[i]})\right)\,,
\end{equation}
and hence the lemma follows by the following generalisation of Theorem \ref{EGL}.
\end{proof}

Let $q\colon \S \rightarrow B$ proper and and smooth relative dimension $2$. Write
\[q^{[n]}\colon \S_B^{[n]}=\hilb^n(\S/B)\rightarrow B \quad\mbox{and}\quad q^n\colon \S_B^n=\S\times_B\cdots\times_B\S\rightarrow B \]
for the structure morphisms.

\begin{theorem}\label{EGLrel}
Let $\F_1,\ldots,\F_l$ be vector bundles on $\S$ of respective ranks $r_1,\ldots,r_l$. Let $P$ be a polynomial in the Chern classes of $T_{\S_B^{[n]}/B}$ and the Chern classes of the bundles $\F_i^{[n]}$. Then there is a universal polynomial $Q$, depending only on $P$, of degree $\leq n$ in the classes in $A^*(B)$ of the form $q_*p(T_{\S/B},\F_1,\ldots,\F_l)$, with $p$ a polynomial in the Chern classes of the bundles in the brackets and the ranks $r_1,\ldots,r_l$, such that we have
\[q^{[n]}_* P = Q\,.\]
\end{theorem}
\begin{proof}
The argument given in \cite{EGL} directly generalises to the relative case. For the case $\S = S\times B$, see also \cite{KT2}, Section 4. In fact, Proposition 3.1 in \cite{EGL} still holds if we replace $S^{[n+1]}\times S^m$ by $\S_B^{[n+1]}\times_B\S_B^m$, the bundle $T_{S^{[i]}}$ by the relative tangent bundle $T_{\S_B^{[n]}/B}$ and the integrals by push-forward to $B$. It follows that we can find a universal polynomial $\widetilde{P}$ in the Chern classes of the sheaves $T_{\S/B}$, $\O_\Delta$ and the $\F_i$, pulled back along the several projections
\[pr_i\colon \S_B^n\rightarrow \S \quad\mbox{and}\quad pr_{ij}\colon \S_B^n\rightarrow \S\times_B \S\,,\]
such that we have an equation
\[q^{[n]}_*P =q^n_* \widetilde{P}\,.\]
By Grothendieck-Riemann-Roch we have
\[ch(\O_\Delta) = \Delta_*(td(-T_{\S/B}))\]
for the projections $pr_i\colon \S\times_B\S\rightarrow \S$. It follows that $\widetilde{P}$ is a polynomial in the Chern classes the bundles $pr_i^*T_{\S/B}$ and $pr_i^*\F_i$ and the classes $p_{ij}^*[\Delta]$. By the excess intersection formula, a product of classes of the latter form is a polynomial in Chern classes of the bundles $pr_i^*T_{\S/B}$, intersected with a product
\[\Delta^{k_1}\times\cdots\times\Delta^{k_m}\colon \S_B^{m}\hookrightarrow \S_B^n\]
of diagonals
\[\Delta^{k_i}\colon\S\hookrightarrow \S_B^{k_i}\]
for integers
\[ \quad k_1,\ldots,k_m\geq 1 \,,\quad k_1+\ldots+k_m=n\,.\]
It follows that $q^n_*\widetilde{P}$ is a sum of classes $q^m_*\widetilde{P}_m$ for $m=1,\ldots,n$ and polynomials $\widetilde{P}_m$ in the Chern classes of the bundles $pr_i^*T_{\S/B}$ and $pr_i^*\F_i$, pulled back along the several projections $\S_B^m\rightarrow \S$. Now use the fact that for classes $\alpha_1,\ldots,\alpha_m\in A^*(\S)$ we have
\[q_*\alpha_1\cdots q_*\alpha_m = (q^m)_*(pr_1^*\alpha_1\cdots pr_m^*\alpha_m)\]
in $A^*(B)$.
\end{proof}

\begin{proof}[Proof of Proposition \ref{propUniversal}]
By Lemma \ref{lemmaUniversal}, the class $\gamma(\C)$ can be expressed universally as polynomial $\gamma$ in classes $\epsilon(a,b,c)$ of degree $\leq\delta$. Now let $\C$ be the universal curve in a complete linear system $|L|$ on a surface $S$, and let $\PP^\delta \subset |L|$ be a general linear system. Let $\omega\in A^1(\PP^\delta)$ be the class of a hyperplane. As explained in the proof of \cite[Thm. 4.1]{KST}, the algorithm of \cite{EGL} applied to the right hand side of $\eqref{eqInt}$ for $i=\delta$ produces a term $c_2(S)^\delta/\delta!$ coming from the term
\[c_{2\delta}(T_{S^{[\delta]}})\omega^\delta = c_{2\delta}(T_{S_{\PP^\delta}^{[\delta]}/\PP^\delta})\omega^\delta\,.\] As noted in Remark \ref{reKST}, the integrals in \cite{KST} differ by a factor $c(T_{\PP^\delta})$. However, this does not affect the term $c_{2\delta}(T_{S^{[\delta]}})\omega^\delta$. It follows that $\gamma$ is a polynomial of degree $\delta$ in classes the classes $\epsilon(a,b,c)$.
\end{proof}

\begin{proof}[Proof of Theorem \ref{Theorem A}]
Combine Proposition \ref{sup} and Proposition \ref{propUniversal}. We have completed the proof of our first main result.
\end{proof}

\subsection{Multiplicativity}
We will check that the class $\gamma$ has the expected multiplicative behaviour, cf. \cite{KP2} and \cite{Go}. Let $B$ be a base scheme. For $k=1,2$, let $\S_k\rightarrow B$ be proper and smooth of relative dimension $2$, and let $\C_k$ be a relative effective divisor on $\S_k$, and write $p_k\colon \C_k\rightarrow B$ for the morphism to $B$. Let $\C$ be the union
\[\C\coloneqq \C_1\amalg\C_2\subset  \S_1\amalg\S_2\rightarrow B\,.\]
We have the following relations:
\begin{align}
\C_B^{[n]}					&= \coprod_{i+j=n} (\C_1)_B^{[i]} \times_B (\C_2)_B^{[j]}\,;	\nonumber\\
p^{[i]}_*c(T_{\C_B^{[n]}/B})	&= \sum_{i+j=n} p^{[i]}_{1*}c(T_{(\C_1)_B^{[i]}/B}) \cdot p^{[i]}_{2*}c(T_{(\C_2)_B^{[j]}/B}) \label{prela}\,.
\end{align}
For $k=1,2$, let $g_k$ be the arithmetic genus of a curve in the family $\C_k\rightarrow B$, so we have
\begin{equation*}
g-1=g_1-1+g_2-1\,,
\end{equation*}
with $g$ the genus of a curve in the family $\C\rightarrow B$. It follows easily from \eqref{prela} that we have the identity
\begin{equation}\label{prodnvir}
n^\vir_{g-\delta}(\C) = \sum_{i+j=\delta} n^\vir_{g_1-i}(\C_1)\, n^\vir_{g_2-j}(\C_2)\,.
\end{equation}
For any $i\geq 0$, let $\gamma_i$ be the degree-$i$ part of $n^\vir_{g-i}$. We record the following lemma.

\begin{lemma}\label{prodgamma}
Let $B$ be complete and let $\C_1$ and $\C_2$ be given as above. Assume that for $k=1,2$, the line bundle $\O(\C_k)$ on $\S_k$ satisfies $\amp$. Then we have the relation
\[\gamma_\delta(\C) = \sum_{i+j=\delta} \gamma_i(\C_1)\,\gamma_j(\C_2)\]
\end{lemma}
\begin{proof}
By Proposition \ref{sup}, it follows directly from \eqref{prodnvir}, by taking degree-$\delta$ parts on both sides of the equation.
\end{proof}

\section{Application: Plane Curves in $\PP^3$}
We will apply the results to the problem of counting $\delta$-nodal plane curves of degree $d$ in $\PP^3$. As we will see below, the space of such curves has dimension
\begin{equation}\label{defn}
n \coloneqq \frac{d(d+3)}{2} + 3 - \delta\,.
\end{equation}
Let $N_{\delta,d}$ denote the number of $\delta$-nodal plane curves of degree $d$ that intersect general lines $\ell_1,\ldots,\ell_n\subset \PP^3$. The main result of this section is that for each $\delta$, and $d\geq\delta$, the numbers $N_{\delta,d}$ are given by polynomial of degree $\leq 2\delta + 9$ in $d$.

Let $\Gr\coloneqq\Gr(2,\PP^3)$ be the Grassmannian of planes in $\PP^3$ and let $\U$ be the tautological vector bundle on $\Gr$. Let $\O_\Gr(1)$ be the bundle corresponding to the hyperplane class via the identification $\Gr=\check{\PP}^3$. These two bundles are related via the tautological short exact sequence
\begin{equation}\label{taut}
0\rightarrow\U\rightarrow\CC^4\otimes\O_\Gr\rightarrow \O_\Gr(1)\rightarrow 0\,.
\end{equation}
Let $q\colon\S=\PP(\U)\rightarrow \Gr$ be the universal plane. As a family of subvarieties of $\PP^3$, it comes with a relatively very ample line bundle $\O_\S(1)$. Choose an integer $d>1$ and consider the line bundle $\L\coloneqq\O_\S(d)$. As in Section \ref{relkst}, we can form the projective bundle
\[B\coloneqq |\L/\Gr| = \PP(q_*\L)\xrightarrow{\mathmakebox[1cm]{\pi}}\Gr\,,\]
which comes with a canonical bundle $\O_B(1)$. The fibre over a point $[V]\in\Gr$ corresponding to a plane $V\subset \PP^3$ is the complete linear system $|\O_V(d)|$. In particular $\pi$ is of relative dimension $r-1$, with
\[r = \mbox{rank}(q_*\L) = \frac{(d+1)(d+2)}{2}\,.\]
Moreover, it follows that $B$ parametrizes planes in $\PP^3$, together with a degree $d$ curve on that plane. As a planar curve in $\PP^3$ of degree $>1$ lies in a unique plane, the variety in fact parametrizes planar curves in $\PP^3$.

Let $p\colon\C\rightarrow B$ be the universal curve. Then $\C$ is a relative effective divisor on the family $\S_B=\S\times_\Gr B$ and we have
\[\O_{\S_B}(\C) = \L(1) = \L \otimes\O_B(1)\,.\]
Now let $\delta \leq d$. Note that for $b\in B$, we have
\[\S_b\cong \PP^2 \quad \mbox{and} \quad\L(1)|_{\S_b}\cong \O_{\PP^2}(d)\,,\]
so $\L(1)$ satisfies $\amp$, as $\O(d)$ is $\delta$-very ample \cite{BS}. By Lemma \ref{keyex}, the locus $B(\delta)\subset B$ has codimension $\delta$ and we have $\gamma(\C) = \left[\overline{B(\delta)}\right]\subset B$. For lines $\ell_1,\ldots,\ell_n$ in $\PP^3$, denote the (reduced) locus of curves intersecting the line $\ell_i$ by $B_{\ell_i}\subset B$, and let
\[B_{\ell_1,\ldots,\ell_n}=B_{\ell_1}\cap\ldots\cap B_{\ell_n}\subset B\]
be the (scheme theoretic) intersection. We will use the notation
\[\partial B(\delta) = \left.\overline{B(\delta)} \middle\backslash B(\delta)\right.\,.\]

\begin{proposition}\label{finred}
For
\[n = \frac{d(d+3)}{2} + 3 - \delta\]
as above and general lines $\ell_1,\ldots,\ell_n$, we have
\[B_{\ell_1,\ldots,\ell_n}\cap \partial B(\delta) = \emptyset\,.\]
Moreover, the intersection $B_{\ell_1,\ldots,\ell_n}\cap B(\delta)$ is finite and reduced, and its degree is given by
\[N_{\delta,d}=\int_{B_{\ell_1,\ldots,\ell_n}} \gamma(\C|_{B_{\ell_1,\ldots,\ell_n}})\,.\]
\end{proposition}
\begin{remark}
More precisely, in the proof we will construct a non-empty Zariski open subset
\[U\subset \Gr(1,\PP^3)^n\]
of the $n$-fold product of the Grassmannian of lines in $\PP^3$, such that the proposition holds for any $n$-tuple of lines $(\ell_1,\ldots,\ell_n)\in U$.
\end{remark}
\begin{remark}
It should be noted that for $\delta>0$, the scheme $B_{\ell_1,\ldots,\ell_n}$ is singular. In fact, for every $i$, a local computation shows that the singular locus of the variety $B_{\ell_i}$ is the divisor of curves $C\in B_{\ell_i}$ such that $
\ell_i$ lies in the plane spanned by $C$.
\end{remark}
\begin{proof}
We will proof the first two statements by an argument in the spirit of Lemma 4.7 in \cite{KP1}. We have $n = r + 2 - \delta = \dim B -\delta$. It follows that, cf. loc. cit., the expected dimension of $B_{\ell_1,\ldots,\ell_n}\cap B(\delta)$ is
\begin{equation}\label{eqExpDim}
\dim B - n - \delta = 0\,.
\end{equation}
Let $\Gr(1,\PP^3)$ be the Grassmannian of lines in $\PP^3$, and let $\LL\rightarrow \Gr(1,\PP^3)$ be the universal line. Let $P$ be the limit of the following diagram:
\begin{equation}\label{finreddiag}
\begin{tikzcd}
& & \overline{B(\delta)}\arrow[d] \\
&\C_B^n \arrow[r]\arrow[d] & B \\
\LL^n \arrow[r]\arrow[d] &(\PP^3)^n\\
\Gr(1,\PP^3)^n \, ,
\end{tikzcd}
\end{equation}
in which we use the notation
\[\C_B^n = \C\times_B\cdots\times_B\C\,.\]
Then $P$ parametrises the following data:
\begin{itemize}
\item lines $\ell_1,\ldots\ell_n \subset \PP^3$;
\item points $p_1,\ldots,p_n\in \PP^3$;
\item a plane $V \subset \PP^3$;
\item a curve $C\in \overline{B(\delta)}$;
\end{itemize}
subject to the following conditions:
\begin{itemize}
\item $C\subset V$;
\item $p_i\in \ell_i$ for $i=1,\ldots,n$;
\item $p_i\in C$ for $i=1,\ldots,n$.
\end{itemize}
The horizontal maps in the diagram are flat with relative dimensions $2n$ and $n$ respectively. Since $\overline{B(\delta)}$ has dimension $n$, it follows that
\[\dim(P) = 4n = \dim(\Gr(1,\PP^3)^n)\,.\]
A curve $C\in B(\delta)$ has a dense smooth open subset $C^\circ$
(as it is nodal). Therefore also the universal curve $\C\rightarrow B$ restricted to $\overline{B(\delta)}$ has this property (the latter being reduced).
As $P\rightarrow \C_B^n|_{\overline{B(\delta)}}$ is smooth (with fibre $\cong(\PP^2)^n$), we see that $P$ is generically smooth.
In particular, the singular locus $P^{\mathrm{sing}}$ of $P$ has dimension $<4n$. Consider the morphism
\[\phi\colon P\rightarrow \Gr(1,\PP^3)^n\,.\]
Let $U_1=\Gr(1,\PP^3)^n \backslash \phi(P^{\mathrm{sing}})$ be the complement of the image of the singular locus in $\Gr(1,\PP^3)^n$. As $\Gr(1,\PP^3)^n$ has dimension $4n$, the open $U_1\subset \Gr(1,\PP^3)^n$ is non-empty. Since $\phi^{-1}(U_1)$ is smooth, there is a non-empty open $U_2\subset U_1$ such that the morphism $\phi$ is finite and reduced over $U_2$. Moreover, the closed subsets
\[Z_i = \{(\vec{\ell}, \vec{p}, V, C)\in P\mid \ell_i\subset V\}\subset P\,, \quad i=0,\ldots,n \]
and
\[P_{\partial B(\delta)}=\{(\vec{\ell}, \vec{p}, V, C)\in P\mid C\in\partial B(\delta)\}\subset P\]
have positive codimension. Therefore the open
\[U=U_2 \cap \left(\Gr(1,\PP^3)^n \middle\backslash \phi\left(\bigcup_{i=1}^n Z_i\cup P_{\partial B(\delta)}\right)\right)\]
is non-empty.

Now let $\vec{\ell}=(\ell_1,\ldots,\ell_n)\in U$ be an $n$-tuple of lines and let
\[P_{\vec{\ell}} = \phi^{-1}(\vec{\ell}\,)\]
be the fibre over $\vec{\ell}$. Consider the morphism
\[\psi\colon P\rightarrow \overline{B(\delta)}\subset B\,.\]
For a point $[C]\in \overline{B(\delta)}$, the fibre of $P_{\vec{\ell}}$ over $[C]$ is the scheme
\begin{equation}\label{eqFibreP}
P_{\vec{\ell}}\cap \psi^{-1}([C]) = \prod_{i=1}^n (\ell_i\cap C)\subset (\PP^3)^n\,.
\end{equation}
Let $V\subset \PP^3$ be the plane spanned by $C$. By definition of $U$, we have $\ell_i\not\subset V$ for the lines $\ell_1,\ldots,\ell_n$ in $\vec{\ell}$. It follows that the intersection \[C\cap \ell_i\subset \ell_i\cap V\]
is a reduced point, if it is non-empty. Hence $P_{\vec{\ell}}$ maps isomorphically to its image in $B$.

On the other hand, the scheme \eqref{eqFibreP} is non-empty if and only if $C$ intersects the lines $\ell_i$. By definition of $U$, it follows that we have
\begin{align*}
\psi(P_{\vec{\ell}})	&= (B_{\ell_1,\ldots,\ell_n}\cap \overline{B(\delta)})_{red}\\
					&= (B_{\ell_1,\ldots,\ell_n}\cap B(\delta))_{red}\,.
\end{align*}

Finally, we will show that in fact
\[\psi(P_{\vec{\ell}})=B_{\ell_1,\ldots,\ell_n}\cap B(\delta)\,.\]
To see this, first note that $\psi(P_{\vec{\ell}})$ lies in the open subset $W\subset B$ consisting of curves $C$ with $\ell_i\not\subset V$ for $i=1,\ldots, n$ and $V$ the plane spanned by $C$. As noted before, for a curve $C\in W$, the scheme $C\cap \ell_i$ is empty or consists of a single reduced point. It follows that over $W$, the scheme
\[\ell_i\times_{\PP^3} \C\]
is mapped isomorphically to its image $B_{\ell_i}$ in $B$, since a \emph{scheme theoretic} fibre of
\[\ell_i\times_{\PP^3} \C_W\rightarrow B_{\ell_i}\cap W\]
is a \emph{reduced} point. We conclude that
\begin{align*}
\psi(P_{\vec{\ell}})	&= \psi\left((\ell_1\times_{\PP^3} \C)\times_B\cdots\times_B(\ell_n\times_{\PP^3} \C )\times_B \overline{B(\delta)}\right)\\
								&= B_{\ell_1}\cap\cdots\cap B_{\ell_n}\cap \overline{B(\delta)}\\
								&= B_{\ell_1,\ldots,\ell_n}\cap B(\delta)\,.								
\end{align*}

As the intersection $B(\delta)\cap B_{\ell_1,\ldots, \ell_n}$ is finite and reduced, it is \emph{transverse}, by \eqref{eqExpDim}. It follows by Lemma \ref{func}, Proposition \ref{sup} and Lemma \ref{keyex} that we have
\begin{align*}
\#(B(\delta)\cap B_{\ell_1,\ldots,\ell_n})
					&= \int_B ([\overline{B(\delta)}].[B_{\ell_1,\ldots,\ell_n}])\\
					&= \int_B \gamma(\C)\cap [B_{\ell_1,\ldots,\ell_n}]\\
					&= \int_{B_{\ell_1,\ldots,\ell_n}} \gamma(\C|_{B_{\ell_1,\ldots,\ell_n}})\,.\qedhere
\end{align*}
\end{proof}

The following lemma is essentially \cite{Zi}, Exercise 3.4. See also \cite{Fu}, Example 3.2.22. For completeness, we will include the proof.

\begin{lemma}\label{linelocus}
For a line $\ell\subset \PP^3$, the closed subvariety $B_\ell\subset B$ is a divisor, cut out by a section of the line bundle
\[\pi^*\O_\Gr(d)\otimes\O_B(1)\,.\]
\end{lemma}
\begin{proof}
It suffices to construct the section outside the codimension two subvariety $Z\subset \Gr$ of planes $[V]\in\Gr$ containing the line $\ell$. Let $U\coloneqq \Gr-Z$ be the complement. Consider the fibre product $\S_\ell$ of the following diagram:
\[\begin{tikzcd}
& \S\arrow[d]\\
\ell \arrow[r] & \PP^3\,.
\end{tikzcd}\]
Then the morphism $\S_\ell\rightarrow\Gr$ has fibre $V\cap \ell\subset \PP^3$ over a point $[V]\in \Gr$. In particular, it restricts to an isomorphism over $U$. Now consider the fibre product $\C_\ell$ of the diagram
\[\begin{tikzcd}
&\C\arrow[d]\\
\ell\arrow[r]& \PP^3\,.
\end{tikzcd}\]
The morphism $\C_\ell\rightarrow B$ has  \emph{scheme theoretic} fibre $C\cap \ell$ over a point $[C]\in B$. It follows that $\C_\ell$ is mapped onto $B_\ell$ and the morphism restricts to an isomorphism over $B_\ell\times_\Gr U$.

As noted before, $\C$ is the zero locus of a canonical section of the line bundle $\O_\S(d)\otimes\O_{B}(1)$ on $\S\times_\Gr B$, in which $\O_\S(d)$ denotes the pull-back of the bundle $\O_{\PP^3}(d)$ on $\PP^3$ to the universal plane $\S$. It follows that $\C_\ell\subset\S_\ell \times_\Gr B$ is cut out by a section of the bundle $\O_\ell(d)\otimes\O_{B}(1)$, in which $\O_\ell(d)$ is the restriction of $\O_{\PP^3}(d)$ to $\PP^1\cong \ell \subset \PP^3$.
Hence, it suffices to show that $\O_{\PP^3}(d)$ equals $\O_\Gr(d)$ when pulled back to $\S_\ell\times_\Gr B$. But this can be seen easily by noting that in the diagram (with Cartesian square)
\[\begin{tikzcd}
\S_\ell \arrow[r]\arrow[d] & \S \arrow[r] \arrow[d] & \Gr \\
\ell \arrow[r] & \PP^3
\end{tikzcd}\]
the fibre $(\S_\ell)_x$ over a point $x\in \ell$ is the inverse image of the divisor on $\Gr$ of planes in $\PP^3$ containing the point $x$.
\end{proof}

The following corollary of Proposition \ref{finred} is the first part of Theorem \ref{Theorem B}, our second main result.

\begin{corollary}\label{polycurve}
For every $\delta \geq 0$, there is polynomial $N_\delta$ of degree $\leq 9 + 2\delta$ such that $N_{\delta,d}=N_\delta(d)$ for $d\geq \delta$.
\end{corollary}
\begin{proof}
By Proposition \ref{finred}, we need to compute the integral
\[\int_{B_{\ell_1,\ldots,\ell_n}} \gamma(\C|_{B_{\ell_1,\ldots,\ell_n}})=\int_{B} \gamma(\C)\cap[B_{\ell_1,\ldots,\ell_n}] \,.\]
Let $H=c_1(\O_\Gr(1))$ and $\xi=c_1(\O_B(1))$. Then, by the Lemma \ref{linelocus}, we have for general lines $\ell_1,\ldots,\ell_n\subset \PP^3$ the equation
\begin{align*}
[B_{\ell_1,\ldots,\ell_n}]	&=(dH+\xi)^n\\
							&=\sum_{i=0}^3 \binom{n}{i}(dH)^i\xi^{n-i}\,
\end{align*}
in $A^*(B)$.
On the other hand, we know by Theorem \ref{Theorem A} that $\gamma(\C)$ is a polynomial of degree $\delta$ in classes
\[\epsilon(a,b,c) = (q_B)_*(c_1(\O(\C)^a c_1(T_{\S/\Gr})^b c_2(T_{\S/\Gr})^c)\,.\]
It will follow that $\gamma(\C)$ is a polynomial in $H,\xi$ and $d$. To see this, if suffices to show that the classes $\epsilon(a,b,c)$ are polynomials in $H,\xi$ and $d$. Let $\eta = c_1(\O_\S(1))$. Then we have
\begin{align*}
c_1(\O(\C)^a c_1(T_{\S/\Gr})^b c_2(T_{\S/\Gr})^c	&= (d\eta + \xi)^a (3\eta + c_1(\U))^b (3\eta^2 + 2\eta c_1(\U) + c_2(\U))^c\\
																				&= (d\eta + \xi)^a (3\eta - H)^b (3\eta^2 - 2\eta H + H^2)^c\,,
\end{align*}
in which we use the short exact sequences
\[0\rightarrow \O\rightarrow \U\otimes \O_\S(1)\rightarrow T_{\S/\Gr}\rightarrow 0\]
and \eqref{taut}. The structure of the Chow ring of $\S_B$ is given by
\[A^*(\S)=A^*(B)[\eta]/(\eta^3+ c_1(\U)\eta^2 + c_2(\U) \eta + c_3(\U))\]
and the push-forward
\[\epsilon(a,b,c) = (q_B)_*(c_1(\O(\C)^a c_1(T_{\S/\Gr})^b c_2(T_{\S/\Gr})^c)\]
is computed by repeatedly substituting the equation
\begin{align*}
\eta^3 &= - c_1(\U)\eta^2 - c_2(\U) \eta -c_3(\U)\\
			&= H \eta^2 - H^2 \eta + H^3\,.
\end{align*}
and taking the coefficient of $\eta^2$. Since $H^4=0$, the substitution procedure terminates after $3$ steps. For fixed $a,b$ an $c$, we obtain a polynomial in $H,\xi$ and $d$. It follows that $\gamma(\C)$ can be written as a polynomial in $H, \xi$ and $d$.

Note that the coefficient of $H^i$ in $\epsilon(a,b,c)$ has degree at most $2+i$ as a polynomial in $d$. As $\gamma(\C)$ is a polynomial of degree $\delta$ in classes $\epsilon(a,b,c)$, the coefficient of $H^i$ in this $\gamma(\C)$ has degree at most $2\delta + i$, as a polynomial in $d$.

We consider the class
\[\gamma(\C|_{B_{\ell_1,\ldots,\ell_n}})=\gamma(\C)\cap [B_{\ell_1,\ldots,\ell_n}]\]
in $A_0(B)$. We will show that its degree is a polynomial $d$.
The Chow ring of $B=\PP(q_*\L)$ is given by
\begin{align*}
A_*(B)	& = A^*(\Gr)[\xi]/(\xi^r + c_1(q_*\L)\xi^{r-1} + c_2(q_*\L) \xi^{r-2} + c_3(q_*\L)\xi^{r-3})\\
		& = \ZZ[H,\xi]/(H^4, \xi^r + c_1(q_*\L)\xi^{r-1} + c_2(q_*\L) \xi^{r-2} + c_3(q_*\L)\xi^{r-3})\,.
\end{align*}
We have
\[q_*\L=\sym^d(\U^*)\,,\]
and hence its Chern class is a polynomial in $d$ and $H$. In fact, we have
\begin{multline*}
c(q_*\L) = 1 + {\frac{d \left( d+1 \right)  \left( d+2 \right)}{6}}\, H \,+ \\
{\frac { d \left( d+1 \right)  \left( d+2 \right) \left( d+3 \right)  \left( {d}^{2}+2 \right)}{72} \, {H}^{2}}\, + \\
{\frac{d\left(d+1\right)\left(d+2\right)\left(d+3\right)\left({d}^{2}+2\right)\left({d}^{3}+3\,{d}^{2}+2\,d+12\right)}{1296}\,{H}^{3}}\,.
\end{multline*}
Note that coefficient of $H^i$ is a polynomial in $d$ of degree $3i$. We can compute the degree
\begin{align*}
\int_B\gamma(\C|_{B_{\ell_1,\ldots,\ell_n}})	&= \int_B\gamma(\C)\cap\sum_{i=0}^3 \binom{n}{i}(dH)^i\xi^{n-i}
\end{align*}
as follows. Note that the class $\gamma(\C|_{B_{\ell_1,\ldots,\ell_n}})$ is homogeneous of degree $r+2$ in $H$ and $\xi$. Using the relations
\[\xi^r = - c_1(q_*\L)\xi^{r-1} - c_2(q_*\L) \xi^{r-2} - c_3(q_*\L)\xi^{r-3}\quad \mbox{and}\quad H^4 =0\]
we can rewrite it as
\[u H^3 \xi^{r-1}\]
for a polynomial $u$ in $d$ depending only on $\delta$. Now we have
\[\int_B\gamma(\C|_{B_{\ell_1,\ldots,\ell_n}})=u\,.\]

Recall that we write $\pi\colon B\rightarrow \Gr$ for the projection. Then the class
\[u H^3=\pi_*(u H^3\xi^{r-1})\]
is a product of classes of the following types:
\begin{itemize}
\item The coefficients of $\gamma(\C)$ as a polynomial in $\xi$;
\item Classes of the form $\binom{n}{i}(dH)^i$;
\item Polynomials in the Chern classes of $q_*\L$.
\end{itemize}
As remarked before, the coefficient of $H^i$ in $\gamma(\C)$ has degree $2\delta + i$ in $d$. In other words, every factor $d^{2\delta + i}$ appearing in the terms of $\gamma(\C)$ is accompanied by a factor $H^i$. Similarly, the coefficients of $H^i$ of classes of the second and third type, are polynomials of degree $3i$ in  $d$, so every factor $d^{j}$ appearing in the terms of these classes is accompanied by a factor $H^{\lceil j/3\rceil}$. It follows that $u H^3$ has degree at most $2\delta + 9$ in $d$.
\end{proof}

\section{Torus localization}
As in the previous section, let $\Gr=\Gr(2,\PP^3)$ be the Grassmannian of planes in $\PP^3$, with universal plane $q\colon\S=\PP(\U)\rightarrow \Gr$, in which $\U$ is the tautological vector bundle on $\Gr$. On $\S$ we have defined the line bundle $\L=\O_\S(d)$, which we use to construct the projective bundle $B=\PP(q_*\L)$ over $\Gr$ parametrizing planar curves of degree $d$ in $\PP^3$, with universal curve $\C\rightarrow B$. The variety $B$ has dimension $r + 2$, with
\[r = \frac{(d+1)(d+2)}{2}\]
the rank of $q_*\L$. Finally, we define
\[n\coloneqq r + 2 - \delta = \frac{d(d+3)}{2} + 3 - \delta\]
and write $B_{\ell_1,\ldots,\ell_n}$ for the locus of curves intersecting general lines $\ell_1,\ldots,\ell_n$.

Rather than using the algorithm of \cite{EGL}, we can use the Bott residue formula to evaluate the integral
\[\int_B \gamma(\C|_{B_{\ell_1,\ldots,\ell_n}})=\int_B \gamma(\C)\cap [{B_{\ell_1,\ldots,\ell_n}}]\,.\]
of Proposition \ref{finred}. By the Lemmas \ref{linelocus} and \ref{lci}, we need to compute
\begin{equation}\label{integraaltje}
\begin{split}
&\int_{\C_B^{[i]}}c(T_{\C_B^{[i]}/B})\cap \left[\C_{B_{\ell_1,\ldots,\ell_n}}^{[i]}\right]=\\
&\int_{\S_B^{[i]}}\frac{c(T_{\S_B^{[i]}/B})}{c(\O(\C)_B^{[i]})}\, c_i(\O(\C)_B^{[i]})\, c_1(\O_\Gr(d)\otimes\O_B(1))^n\\
\end{split}
\end{equation}
for $i=1,\ldots,\delta$. We have equations
\[\O_{\S_B}(\C)_B^{[i]}=(\L\otimes\O_B(1))_B^{[i]}=\L_\Gr^{[i]}\otimes\O_B(1)\]
and
\[T_{\S_B^{[i]}/B}=\pi_{\S^{[i]}}^*T_{\S_\Gr^{[i]}/\Gr}\,.\]
It follows that the class on the right hand side of \eqref{integraaltje} is a polynomial in classes pulled back from $\S_\Gr^{[i]}$, and the first Chern class of the line bundle $\O_B(1)$. We will continue to use the notation $\xi=c_1(\O_B(1))$ and $H=c_1(\O_\Gr(1))$. We can rewrite the factors involving the bundle $\O_B(1)$ in the integral as follows:
\begin{align*}
c_1(\O_\Gr(d)\otimes\O_B(1))^n				&=(d H + \xi)^n\\
											&=\left(\xi^3 + n\,dH\,\xi^2 + \binom{n}{2}\,(dH)^2\,\xi + \binom{n}{3}(dH)^3\right)\,\xi^{n-3}\\
c_i(\L_\Gr^{[i]}\otimes\O_B(1))				&=\sum_{k=0}^i c_{k}(\L_\Gr^{[i]})\,\xi^{i-k}\\
\frac{1}{c(\L_\Gr^{[i]}\otimes\O_B(1))}		&=\sum_{j=0}^\infty \left(1-c(\L_\Gr^{[i]}\otimes\O_B(1))\right)^j\\
											&=\sum_{j=0}^\infty \left(1-\sum_{k=0}^i c_k(\L_\Gr^{[i]})\,(1+\xi)^{i-k}\right)^j\\
											&=\sum_{j=0}^{3+2i+\delta} \left(1-\sum_{k=0}^i c_k(\L_\Gr^{[i]})\,(1+\xi)^{i-k}\right)^j + \alpha\,.
\end{align*}
In the last expression, $\alpha\in A^*(\S_B^{[i]})$ is a sum of classes of degree $>3+2i+\delta$. It follows that
\[\alpha\, \xi^{n-3}=0\,,\]
as we have
\[(n-3)+(3+2i+\delta) = r + 2 + 2i = \dim(\S_B^{[i]})\,.\]

Hence we can compute \eqref{integraaltje} by integrating the class
\begin{multline*}
q^{[i]}_*\Bigg(
c(T_{\S_\Gr^{[i]}/\Gr})\quad\times\quad
\left(\xi^3 + n\,dH\,\xi^2 + \binom{n}{2}\,(dH)^2\,\xi + \binom{n}{3}(dH)^3\right)\,\xi^{n-3}\quad\times\\
\sum_{k=0}^i c_{k}(\L_\Gr^{[i]})\,\xi^{i-k}\quad\times\quad
\sum_{j=0}^{3+2i+\delta} \left(1-\sum_{k=0}^i c_k(\L_\Gr^{[i]})\,(1+\xi)^{i-k}\right)^j
\Bigg)\,.
\end{multline*}
As in the proof of Corollary \ref{polycurve}, the push-forward along the projection
\[\pi\colon \S_B^{[i]} = \PP(q_*\L)\times_\Gr \S_\Gr^{[i]} \rightarrow \S_\Gr^{[i]}\]
can be computed by substituting the equation
\[\xi^r = - c_1(q_*\L)\xi^{r-1} - c_2(q_*\L) \xi^{r-2} - c_3(q_*\L)\xi^{r-3}\]
and taking the coefficient of $\xi^{r-1}$. Hence we can rewrite \eqref{integraaltje} as an integral
\begin{equation}\label{Bottintegraal}
\int_{\S_\Gr^{[i]}} P(T_{\S_\Gr^{[i]}/\Gr},\L_\Gr^{[i]},\O_\Gr(1),d)\,,
\end{equation}
in which $P$ is a polynomial\footnote{The fact that the expression is polynomial in $d$ is not important for the computation. However, it does give another proof of the polynomiality of $N_{\delta,d}$ for $d\geq\delta$, that does not depend on the algorithm of \cite{EGL}.} in $d$ and the Chern classes of the bundles in the brackets.

We can evaluate this integral using the Bott residue formula. For notation and definitions, see \cite{EG}. Recall that for a torus $T$ acting on a smooth variety $X$, the fixed locus $X^T$ is smooth \cite{Iv}, so a connected component $F\subset X^T$ has normal bundle $N_FX$ of rank equal to the codimension $d_F$ of $F$ in $X$.)
\begin{theorem}[Bott residue formula (\cite{EG})]\label{Bott}
Let $E_1,\ldots,E_r$ be $T$-equivariant vector bundles on a complete, smooth $n$-dimensional variety $X$ with a torus action by $T$. Let $p(E)$ be a polynomial in the Chern classes of the bundles $E_i$. Then
\[\int_X p(E)\cap [X] = \sum_{F \subset X^T} \int_F \left(\frac{p^T(E|_F)\cap[F]_T}{c^T_{d_F}(N_FX)}\right)\,,\]
in which we sum over the connected components $F$ of the fixed locus $X^T$ of the torus action.
\end{theorem}

Consider the natural action of the torus $T=(\mathbb{G}_m)^4$ on $\PP^3$ given by
$$(\lambda_0,\lambda_1,\lambda_2,\lambda_3)\cdot(a_0:a_1:a_2:a_3) = (\lambda_0a_0:\lambda_1a_1:\lambda_2a_2:\lambda_3a_3).$$
It induces a dual action of $T$ on the Grassmanian $\Gr = \check{\PP}^3$ which lifts to an equivariant structure on $\O_\Gr(1)$. The action on $\check{\PP^3}\times\PP^3$ restricts to an action on $\S$ (which is simply the incidence variety). This action, in turn, lifts to an equivariant structure on the line bundle $\O_\S(1)$. Moreover, we obtain actions on the Hilbert schemes $\S^{[i]}$ and induced equivariant structures on the the bundels $T_{\S^{[i]}/\Gr}$ and $\L^{[i]}=\O_\S(d)^{[i]}$.

\begin{lemma}\label{fixedloc}
The fixed locus of the action of $T$ on $\S_\Gr^{[i]}$ is finite and reduced.
\end{lemma}
\begin{proof}
As the variety $\S_\Gr^{[i]}$ is smooth, so is the fixed locus, as remarked above. Hence it suffices to show that the underlying set is finite. The fixed points of $\Gr$ are the four planes
\[V_k = \Z(x_i)\subset \PP^3=\proj\CC[x_0,x_1,x_2,x_3]\]
for $k=0,\ldots,3$, given by the vanishing of a coordinate. Note that the morphism $\S_\Gr^{[i]}\rightarrow \Gr$ is equivariant for the $T$-action. It follows that $T$ acts on the fibres $V_k^{[i]}$. Now we follow \cite{ES}, Section 4. Let $Z\subset V_0$ be a subscheme of length $i$, fixed under the action of $T$. Then $Z$ is supported on the $T$-invariant locus $\{P_1, P_2, P_3\}\subset V_0$, with
\[
P_1 = (0:1:0:0),\quad
P_2 = (0:0:1:0),\quad
P_3 = (0:0:0:1)\, .
\]
For $k=1,2,3$, let $Z_k$ the component of $Z$, supported on $\{P_k\}$, and let $i_k$ the length of $Z_k$. On an open neighbourhood of $P_1$, we have coordinates $u=x_2/x_1$ and $v=x_3/x_1$ on the plane $V_0$. Now $T$ acts on the coordinate ring $\CC[u,v]$ by
\[\lambda \cdot u = \frac{\lambda_1}{\lambda_2} u\quad \mbox{and}\quad \lambda \cdot v = \frac{\lambda_1}{\lambda_3} v\,.\]
As $Z$ is invariant, so is the ideal $\mathcal{I}(Z_1)$ of $Z_1$ in $\CC[u,v]$. It follows that $\mathcal{I}(Z_1)$ is generated by monomials in $u$ and $v$. The coordinate ring $\CC[u,v]/\mathcal{I}(Z_1)$ is spanned by the $i_1$ monomials $u^kv^l$ not contained in $\mathcal{I}(Z_1)$. For every $k\geq0$, define
\[d_k = \max \{l \mid u^kv^l\notin \mathcal{I}(Z_1)\}\,.\]
It is easy to see that the $d_k$ define a partition
\[P_{Z_1} = (d_0\geq\ldots\geq d_n)\]
of length $l_1$. Similarly, we get partitions $P_{Z_2}$ and $P_{Z_3}$. Conversely, any \emph{tripartition} $(P_1,P_2,P_3)$ of length $i$, consisting of three partitions $P_1,P_2,P_3$ with $|P_1|+|P_2|+|P_3| = i$, corresponds to a $T$-invariant length-$i$ subscheme of $V_0$. It is clear that there are finitely many such tripartitions. By repeating this argument for the other planes $V_k$, the result follows.
\end{proof}

We will apply Theorem \ref{Bott} to the integral \eqref{Bottintegraal}. By Lemma \ref{fixedloc}, the fixed locus consists of isolated points. Hence, for a fixed point $[Z]\in\S_\Gr^{[i]}$, the normal bundle $N_{[Z]}\S_\Gr^{[i]}$ is just the restriction of the tangent bundle of $\S_\Gr^{[i]}$. The restrictions of the bundles $T_\Gr$, $T_{\S_\Gr^{[i]}/\Gr}$, $\L^{[i]}$ and $\O_\Gr(1)$ to $[Z]$ are $T$-representations.
%
%
Lemma 3 of \cite{EG} gives the equivariant Chern classes explicitly as polynomials in the characters of the torus. The details are similar to the computation in \cite{ES}.

\begin{proof}[Proof Theorem \ref{Theorem B}, second part]
We have performed the calculation using Maple. We have computed the polynomials $N_\delta$ up to $\delta = 12$. The results are printed in Appendix \ref{Appendix A}.
\end{proof}

\begin{remark}\label{Ritwik}
Up to $\delta = 7$, our answers are in agreement with polynomials communicated by Ritwik Mukherjee, which he calculated using the methods from \cite{BM1} \cite{BM2} and \cite{Zi2017} and verified by means of the algorithm of \cite{KP2}.
\end{remark}

\begin{remark}
This method of calculating node polynomials seems to be quite efficient. For example, consider the integral
\[\int_B \gamma(\C)\,c_1(\O_\Gr(1))^3\cap [{B_{\ell_1,\ldots,\ell_{n-3}}}]\,.\]
It is easy to see that it computes the number of $\delta$-nodal curves of degree $d$ intersecting $n-3$ general points in a \emph{fixed} plane $\PP^2$. By an minor adaptation of our code, we were able to compute the node polynomials up to $\delta = 15$, finding agreement with the polynomials up to $\delta = 14$, published by Block in \cite{Bl}. However, G\"ottsche has computed the polynomials up to $\delta\leq 28$ \cite{Go}. The polynomial for $\delta = 15 $ is given in Appendix \ref{Appendix B}.
\end{remark}

\section{Low degree checks}
Let $\delta, d\geq 1$. We want to determine the contribution of reducible curves to the number $N_{\delta,d}$ of planar $\delta$-nodal curves of degree $d$ in $\PP^3$ intersecting general lines $\ell_1,\ldots,\ell_n\subset{\PP^3}$. For certain $\delta$ and $d$, \emph{all} curves contributing to $N_{\delta,d}$ are reducible, thereby giving consistency checks of our formulae. If in addition the irreducible components of these reducible curves are smooth or $1$-nodal, these can be calculated by classical methods.

Let $C$ be a curve in our counting problem, and assume we can write $C$ as the union of irreducible curves $C=C_1\cup\ldots\cup C_r$. The curves $C_i$ are necessarily nodal (if singular), and intersect transversely. For $i=1,\ldots,r$, let $\delta_i$ be the number of nodes of $C_i$, and $d_i$ its degree. As $C$ lies in a plane, two curves $C_i$ and $C_j$  with $1 \leq i < j \leq r$ intersect in $d_id_j$ points. We have
\begin{align}\label{data}
\begin{aligned}
d &= \sum_{i=1}^r d_i\, ,\\
\delta &= \sum_{1\leq i<j \leq r} d_i d_j + \sum_{i=1}^r\delta_i\, .
\end{aligned}
\end{align}
Moreover, there is a partition
\begin{equation}\label{partition}
\{\ell_1,\ldots,\ell_n\} = \bigsqcup_{i=1}^r\Sigma_i
\end{equation}
such that $C_i$ intersects the lines in $\Sigma_i\subset\{\ell_1,\ldots,\ell_n\}$.

Conversely, choose a partition as in (\ref{partition}) and integers $d_i$ and $\delta_i$ for $i=1,\ldots,r$ such that the equations (\ref{data}) hold. We will determine the number of curves contributing to $N_{\delta,d}$ that decompose as described above, with these fixed data.

For $i\in\{1,\ldots,r\}$, let $B_i \coloneqq \PP(\sym^{d_i}(\U^*))\xrightarrow{\pi_i}\Gr=\Gr(2,\PP^3)$ be the projective bundle parametrizing planar curves of degree $d_i$. Consider the locus
\[W_i\coloneqq W(d_i,\delta_i,\Sigma_i)\subset B_i\]
of \emph{irreducible} $\delta_i$-nodal curves that intersect the lines in $\Sigma_i$.

We have the following lemma.

\begin{lemma}\label{decomp}
For general $\ell_1,\ldots,\ell_n$ the number of curves that contribute to $N_{\delta,d}$ that decompose with data fixed above, is given by
$$\int_\Gr \, \prod_{i=1}^r\pi_{i*}[\overline{W}_i]\, ,$$
in which $[\overline{W}_i]$ denotes the class in $A_*B_i$ of the closure of $W_i$.
\end{lemma}
\begin{remark}
More precisely, in the proof we will construct a non-empty Zariski open
\[U\subset \Gr(1,\PP^3)^n\]
such that for an $n$-tuple $(\ell_1,\ldots,\ell_n)\in U$, the statement of the lemma holds.
\end{remark}
\begin{proof}
The argument is similar to the proof of Proposition \ref{finred}, so we will not give all the details. For $i=1,\ldots,r$, let $n_i=\#\Sigma_i$ and form the limit $P_i$ as in the diagram \eqref{finreddiag}. We have natural morphisms $\phi_i\colon P_i\rightarrow \Gr(1,\PP^3)^{n_i}$. Consider the morphism
\[\phi = (\phi_1,\ldots,\phi_r)\colon P_1\times_\Gr\cdots\times_\Gr P_r\rightarrow \prod_{i=1}^r\Gr(1,\PP^3)^{n_i} = \Gr(1,\PP^3)^n\,.\]
As in Proposition \ref{finred}, $\phi$ is finite and smooth over a non-empty open
\[U_0\subset\Gr(1,\PP^3)^n\,.\]
Here we use the fact due to Severi, that for the line bundles $\O(d)$ on $\PP^2$, the locus of irreducible $\delta$-nodal curves in $|\O(d)|$, if non-empty, has codimension $\delta$ \cite{Se}. There is a non-empty open $U_1\subset U_0$, such that for a point $\Sigma=(\Sigma_1,\ldots,\Sigma_r)\in U_1$, the fibre over $\Sigma$ is
\begin{align*}
\phi^{-1}(\Sigma)	&= \overline{W}_1\times_\Gr\cdots\times_\Gr \overline{W}_r\\
								&= W_1\times_\Gr\cdots\times_\Gr W_r\,.
\end{align*}
Finally, there is an non-empty open $U_2\subset U_1$, such that for $\Sigma\in U_2$, and any point $(C_1,\ldots,C_r)\in \phi^{-1}(\Sigma)$, the curves $C_1,\ldots,C_r$ intersect transversely, i.e.
\[C=C_1\cup\ldots\cup C_r\subset V\]
is a nodal curve, in which the union is taken in the plane $V\subset \PP^3$ corresponding to the image of $(C_1,\ldots,C_r)$ in $\Gr$. By a count of dimensions, the sets $\pi_i(W_i)$ intersect properly in $\Gr$. It follows that the contribution to $N_{\delta,d}$ by curves of this type is given by
\begin{align*}
\#\phi^{-1}(\Sigma)		&= \int_{P_1\times_\Gr\cdots\times_\Gr P_r} \, [\overline{W}_1\times_\Gr\cdots\times_\Gr \overline{W}_r]\\
									&= \int_\Gr \, \prod_{i=1}^r\pi_{i*}[\overline{W}_i]\,. \qedhere
\end{align*}
\end{proof}

\begin{notation}
Let $\ell_1,\ldots,\ell_n$ general lines in $\PP^3$, and let $\Sigma_i\subset \{\ell_1,\ldots,\ell_n\}$ be a subset with $\#\Sigma_i=n_i$. Let $W_i=W(d_i,\delta_i,\Sigma_i)$ be the locus in $B_i=\PP(\sym^{d_i} (\U^*))\xrightarrow{\pi}\Gr = \Gr(2,\PP^3)$ of irreducible $\delta_i$-nodal degree $d_i$ plane curves in $\PP^3$ intersecting the lines in $\Sigma_i$. We will write $\nu_{d_i,\delta_i,n_i}$ for the class $\pi_*[\overline{W}_i]\in A_*(\Gr)$.
\end{notation}

We formulate the conclusion of the discussion above in the following proposition.

\begin{proposition}\label{componentcount}
The number of $\delta$-nodal plane curves in $\PP^3$ of degree $d$, intersecting $n=\frac{d(d+3)}{2}+3-\delta$ general lines $\Sigma=\{\ell_1,\ldots,\ell_n\}$, is given by
\[N_{\delta,d} = \sum_{r=1}^\infty\sum_{(\bar{d}, \bar{\delta}, \bar{n})} \mu(\bar{n}) \, \nu_{d_1,\delta_1,n_1}\cdots\nu_{d_r,\delta_r,n_r}\]
in which the second sum is taken over triples of \emph{unordered} $r$-tuples (or multisets)
\[\bar{d}=(d_1,\ldots,d_r)\,,\quad \bar{\delta}=(\delta_1,\ldots,\delta_r)\quad\mbox{and}\quad \bar{n}=(n_1,\ldots,n_r)\]
of integers $d_i\geq1$, $\delta_i\geq 0$, $n_i\geq 0$ satisfying \eqref{data} and
\[n = n_1+\ldots+n_r\,,\]
and in which the multiplicity $\mu(\bar{n})$ is the number of unordered partitions of the set $\Sigma$ in sets of lengths $n_1,\ldots,n_r$. In fact, this number is given by
\[\mu(\bar{n})=\frac{1}{\#\mathrm{Stab}_{S_r}(n_1,\ldots,n_r)}\binom{n}{n_1,\ldots,n_r}\]
in which the denominator is the order of the stabilizer subgroup of $(n_1,\ldots,n_r)\in\ZZ^r$ for the action of the symmetric group $S_r$ on $\ZZ^r$.
\end{proposition}

\begin{remark}
The generality condition in the proposition means that the lines have to be general in the sense of Lemma \ref{decomp}, for every triple $(\bar{d},\bar{\delta},\bar{n})$ appearing in the second sum.
\end{remark}

Using the proposition, we will compute the numbers $N_{\delta,d}$, for $0\leq\delta \leq 6$ and $\delta = 8$ and certain low $d$. We will compare results with with the numbers $N_\delta(d)$, with $N_\delta$ the node polynomial as computed in the previous section, and given in the appendix for $\delta\leq 12$. In these cases, we can choose $d$ in such a way that the irreducible components of the curves are smooth or $1$-nodal.

\begin{lemma}
For the following $\delta$ and $d$, the irreducible components of $\delta$-nodal plane curves of degree $d$ of are lines.

\vbox{\begin{center}
\bgroup
\def\arraystretch{1.2}
\captionof{table}{Nodal plane curves consisting of lines in $\PP^3$}\label{tabeleen}
\begin{tabular}{| c | r | r | r |}
\hline
$\delta$ & $1$ & $3$ & $6$ \\ \hline
$d$ & $2$ & $3$ & $4$ \\ \hline
$N_\delta(d)$ & $140$ & $7280$ & $261800^*$\\ \hline
\end{tabular}
\egroup
\end{center}}
\vspace{1em}
In the following cases, a $\delta$-nodal plane curve of degree $d$ has only linear components besides one smooth conic component.

\vbox{\begin{center}
\bgroup
\def\arraystretch{1.2}
\captionof{table}{Nodal plane curves with a smooth conic component}\label{tabeltwee}
\begin{tabular}{| c | r | r | r |}
\hline
$\delta$ & $0$ & $2$ & $5$ \\ \hline
$d$ & $2$ & $3$ & $4$ \\ \hline
$N_\delta(d)$ & $92$ & $15660$ & $1303500^*$\\ \hline
\end{tabular}
\egroup
\end{center}}
\vspace{1em}
Finally, $\delta$-nodal plane curves of degree $d$ of the following types have only linear components besides two smooth conic components, or a nodal cubic component.

\vbox{\begin{center}
\bgroup
\def\arraystretch{1.2}
\captionof{table}{Nodal plane curves with two conics or a nodal cubic.}\label{tabeldrie}
\begin{tabular}{| c | r | r |}
\hline
$\delta$ & $4$ & $8$ \\ \hline
$d$ & $4$ & $5$ \\ \hline
$N_\delta(d)$ & $3071796$ & $385022820^*$\\ \hline
\end{tabular}
\egroup
\end{center}}
\vspace{1em}
\end{lemma}
\begin{remark}
Since we can apply Theorem \ref{Theorem A} only under the assumption $d\geq \delta$, we have not proved that the value of the polynomial $N_\delta(d)$ equals the curve count $N_{\delta,d}$ in the cases indicated with *. However, as we will prove below, in these cases the polynomials give the right numbers. In general, the \emph{G\"ottsche threshold} $d \geq\lceil \delta/2 \rceil+1$ for nodal curves in $\PP^2$, determined in \cite{KS}, seems to hold also in our case, i.e.~that the node polynomials $N_\delta$ have value $N_{\delta,d}$ in these $d$.
\end{remark}

\begin{proof}
For an integral curve $C$, and its normalisation $\widetilde{C}$, we have
\[g(C) - \delta(C) = g(\widetilde{C})\geq 0\,.\]
It follows that the number of nodes of an irreducible plane curve of degree $d$ is bounded by its arithmetic genus $\frac{(d-1)(d-2)}{2}$. Now use the equations \eqref{data}.
\end{proof}

\begin{lemma}\label{LeShuLoci}
Let $H$ be the hyperplane class in $\Gr \cong \check{\PP}^3$. Then we have

\[
\begin{array}{l l l l l l}
\nu_{1,0,2} &= 1 &\nu_{2,0,5} &= 1 &\nu_{3,1,8} &= 12\\
\nu_{1,0,3} &= 2H &\nu_{2,0,6} &= 8H &\nu_{3,1,9} &= 216H\\
\nu_{1,0,4} &= 2H^2 &\nu_{2,0,7} &= 34H^2 &\nu_{3,1,10} &= 2040H^2\\
\nu_{1,0,5} &= 0 &\nu_{2,0,8} &= 92H^3 &\nu_{3,1,11} &= 12960H^3
\end{array}
\]
\end{lemma}

\begin{proof}
Let $n:= \frac{d(d+3)}{2} + 3 - \delta$.
First note that for $d\geq \delta+2$, all $\delta$-nodal curves of degree $d$ are irreducible, so we have by Lemma \ref{keyex}, Proposition \ref{sup} and a slightly adapted version of Proposition \ref{finred} the identity
\[\nu(d,\delta,n-i) = \pi_*(\gamma(\C|_{B_{\ell_1,\ldots,\ell_{n-i}}}))\in A_i(\Gr)\,.\]
Hence we can compute the classes by the methods of the previous sections. In the case that $\delta = 0, 1$, however, the classes can be computed by elementary means.
Let $\delta = 0$. The locus of curves in $B = \PP(\sym^d(\U^*))$ intersecting a line, is cut out by a section of $\O_\Gr(d)\otimes \O_B(1)$. Note that a general such curve is smooth. It follows that we have the equation
\[\nu_{d,0,n-i} = \pi_*(c_1(\O_\Gr(d)\otimes \O_B(1))^{n-i})\,,\]
the right hand side of which can easily be calculated.

Now let $\delta = 1$. For a curve $C\subset \PP^2$, given by a degree $d$ polynomial $f$, the singular locus is given by the equations
\[df = \frac{\partial f}{\partial x_1} dx_1 + \frac{\partial f}{\partial x_2} dx_2 + \frac{\partial f}{\partial x_3} dx_3 = 0\,, \quad f = 0\,.\]
The rules $f \mapsto df$ and $f\mapsto f$ define homomorphisms
\begin{align*}
\O_B(-1)\rightarrow \Omega_{\S/\Gr}\otimes \O_\S(d)\quad\mbox{and}\quad
\O_B(-1)\rightarrow \O_\S(d)
\end{align*}
of bundles on $\S\times_\Gr B$, in which $\S = \PP(\U)\rightarrow \Gr$ is the universal $\PP^2$-bundle over the Grassmannian. The homomorphisms simultaneously vanish on the singular locus of the fibres of the universal curve $\C\rightarrow B$. As curves with one node lie dense in this locus, it follows that the class of the closure of the locus of 1-nodal curves in $B$ is given by
\[ \alpha = (pr_B)_*(c_3((\Omega_{\S/\Gr}\oplus \O_{\S\times_\Gr B})\otimes\O_\S(d)\otimes\O_B(1))\in A^1(B)\,.\]
We conclude that
\[\nu_{d,1,n-i}=\pi_*(c_1(\O_\Gr(d)\otimes \O_B(1))^{n-i}\cap \alpha)\,.\]
Again, by a straight-forward computation, we obtain the numbers in the third column.
\end{proof}

For $n= \frac{d(d+3)}{2} + 3 - \delta$, the number
\[\int \nu_{d,\delta,n-i} \,  H^i\]
has the following interpretation: it is the number of planar curves $C$ in $\PP^3$ of degree $d$, with $\delta$ nodes, intersecting general lines $\ell_1,\ldots,\ell_{n-i}\subset \PP^3$, such that the plane of the curve contains general points $P_1,\ldots,P_i\in \PP^3$. For certain cases, this enumerative problem has already been studied by Schubert using his calculus introduced in \cite{Schu}. E.g.\ he treats conics, planar and twisted cubics and planar quartic curves in $\PP^N$ that intersect points, lines and planes. The curves are allowed to have nodal singularities, or a cusp in the case of the planar cubic. The degrees of the classes in the second and third column of Lemma \ref{LeShuLoci} can be found in §20 and §24 of loc.\ cit.\ respectively. By Proposition \ref{componentcount} it follows that the numbers in Tables \ref{tabeleen} - \ref{tabeldrie} can computed by 19th century geometry and some elementary combinatorics.

\subsection*{Curves with only linear components}
\begin{itemize}
\item {\boldmath$\delta = 2, d = 2$}
\[N_{2,2} =\binom{7}{3,4}\times 2^2=140=N_2(2)\,.\]
\item {\boldmath$\delta = 3, d = 3$}
\[N_{3,3}=\binom{9}{4,3,2}\times 2^2+\frac{1}{3!}\times\binom{9}{3,3,3} \times 2^3=7280=N_3(3)\,.\]
\item {\boldmath$\delta = 6, d = 4$}
\[N_{6,4}=\frac{1}{2!}\times\binom{11}{4,3,2,2}\times 2^2 + \frac{1}{3!}\times\binom{11}{3,3,3,2}\times 2^3=261800=N_6(4)\,.\]
\end{itemize}

\subsection*{Curves with one conic component}
\begin{itemize}
\item {\boldmath$\delta = 0, d = 2$}
\[N_{0,2}=\nu_{2,0,8}=92=N_2(0)\,.\]
\item {\boldmath$\delta = 2, d = 3$}
\[N_{2,3}=\binom{10}{8}\times 92 + \binom{10}{7}\times 34\times 2+ \binom{10}{6}\times 8\times 2 = 15660=N_2(3)\,.\]
\item {\boldmath$\delta = 5, d = 4$}
\begin{multline*}
N_{5,4}=\frac{1}{2!}\times\binom{12}{8,2,2}\times 92 + \binom{12}{7,3,2}\times 34\times 2 + \binom{12}{6,4,2}\times 8\times 2\,+\\
\frac{1}{2!}\times\binom{12}{6,3,3}\times 8\times 2^2 +\binom{12}{5,4,3}\times 2^2 = 1303500=N_5(4)\,.
\end{multline*}
\end{itemize}

\subsection*{Curves with a nodal cubic or two conic components}
\begin{itemize}
\item {\boldmath$\delta = 4, d = 4$}
The contribution of curves with two conic components is
\[\binom{13}{8}\times 92+\binom{13}{7}\times 34 \times 8 = 585156\,.\]
The contribution of curves consisting of a nodal cubic and a line is
\[\binom{13}{11}\times 12960 + \binom{13}{10} \times 2040 \times 2 + \binom{13}{9}\times 216 \times 2 = 2486640\,.\]
It total we have:
\[N_{4,4}=585156+2486640=3071796=N_4(4)\,.\]
\item {\boldmath$\delta = 8, d = 5$}
Two conics and a line:
\begin{multline*}
\binom{15}{8,5,2} \times 92 + \binom{15}{7,6,2} \times 34 \times 8 + \binom{15}{7,5,3} \times 34 \times 2\,+\\
\frac{1}{2!} \times \binom{15}{6,6,3} \times 8 \times 8 \times 2 + \binom{15}{6,5,4} \times 8 \times 2 = 122942820\,.
\end{multline*}
A nodal cubic, and two lines:
\begin{multline*}
\frac{1}{2!} \times \binom{15}{11,2,2} \times 12960 + \binom{15}{10,3,2} \times 2040 \times 2 + \binom{15}{9,4,2} \times 216 \times 2\,+ \\
\frac{1}{2!} \times \binom{15}{9,3,3} \times 216 \times 2 \times 2 + \binom{15}{8,4,3} \times 12 \times 2 \times 2 = 262080000\,.
\end{multline*}
Total:
\[N_{8,5}=122942820+262080000=385022820=N_8(5)\,.\]
\end{itemize}

\newpage

\appendix
\section{Node polynomials for $\delta = 0,\ldots, 12$.}\label{Appendix A}
In order to keep the denominators under control, we will print the node polynomials for curves with \emph{ordered} nodes, i.e. $N^o_{\delta}=\delta!\,  N_{\delta}$.

{\raggedright\doublespacing\vbadness=10000
\setlength\parindent{-2.8em}

$N^o_{0} = \frac{1}{324}\, d\, (d-1)\, (d+2)\, (d+1)\, (d^{2}+4\, d+6)\, (2\, d^{3}+6\, d^{2}+13\, d+3)$

$N^o_{1} = \frac{1}{108}\, d\, (d+3)\, (d+2)\, (2\, d^{4}+4\, d^{3}+d^{2}-10\, d-6)\, (d-1)^{2}\, (d+1)^{2}$

$N^o_{2} = \frac{1}{108}\, d\, (d-1)\, (d-2)\, (d+2)\, (d+1)\, (6\, d^{8}+30\, d^{7}-25\, d^{6}-255\, d^{5}-142\, d^{4}+333\, d^{3}+629\, d^{2}+18\, d+198)$

$N^o_{3} = \frac{1}{108}\, d\, (d-1)\, (d-2)\, (18\, d^{12}+108\, d^{11}-315\, d^{10}-2664\, d^{9}+470\, d^{8}+21919\, d^{7}+19103\, d^{6}-58136\, d^{5}-106948\, d^{4}+7039\, d^{3}+129360\, d^{2}-165798\, d+110700)$

$N^o_{4} = \frac{1}{36}\, (d-1)\, (d-3)\, (18\, d^{15}+90\, d^{14}-747\, d^{13}-3843\, d^{12}+11660\, d^{11}+63140\, d^{10}-75352\, d^{9}-486678\, d^{8}+73143\, d^{7}+1773729\, d^{6}+1150606\, d^{5}-4123550\, d^{4}-3282032\, d^{3}+12893256\, d^{2}-11795040\, d+3404160)$

$N^o_{5} = \frac{1}{36}\, (d-1)\, (54\, d^{18}-4545\, d^{16}+1152\, d^{15}+159342\, d^{14}-67218\, d^{13}-2985967\, d^{12}+1450512\, d^{11}+32041927\, d^{10}-12936036\, d^{9}-198254910\, d^{8}+9946932\, d^{7}+752976733\, d^{6}+563804514\, d^{5}-2869526338\, d^{4}-1811459616\, d^{3}+11267964504\, d^{2}-12007211040\, d+4224182400)$

$N^o_{6} = \frac{1}{36}\, (162\, d^{21}-486\, d^{20}-17901\, d^{19}+56781\, d^{18}+836361\, d^{17}-2772558\, d^{16}-21438711\, d^{15}+73412631\, d^{14}+327808568\, d^{13}-1138677007\, d^{12}-3072121759\, d^{11}+10302259428\, d^{10}+18632510223\, d^{9}-50159288793\, d^{8}-99732049025\, d^{7}+130703793592\, d^{6}+629801216266\, d^{5}-777706339956\, d^{4}-2089991213304\, d^{3}+5446674186768\, d^{2}-4582360442880\, d+1359752313600)$

$N^o_{7} = \frac{1}{12}\, (162\, d^{23}-810\, d^{22}-22275\, d^{21}+117045\, d^{20}+1315044\, d^{19}-7305633\, d^{18}-43435062\, d^{17}+257593851\, d^{16}+875704283\, d^{15}-5620623440\, d^{14}-11055698265\, d^{13}+77840061643\, d^{12}+89179790228\, d^{11}-672462975543\, d^{10}-563743329044\, d^{9}+3506892852821\, d^{8}+4693983485919\, d^{7}-13574568995962\, d^{6}-37376320692374\, d^{5}+84863008074540\, d^{4}+101290677876264\, d^{3}-419002213496112\, d^{2}+415086981865920\, d-136551736742400)$

$N^o_{8} = \frac{1}{12}\, (486\, d^{25}-3402\, d^{24}-80271\, d^{23}+603369\, d^{22}+5736609\, d^{21}-47210985\, d^{20}-230681484\, d^{19}+2141947278\, d^{18}+5649412578\, d^{17}-62197110162\, d^{16}-84069436618\, d^{15}+1201119124190\, d^{14}+695539180710\, d^{13}-15500834280650\, d^{12}-2727660315107\, d^{11}+131722402261845\, d^{10}+25466213716945\, d^{9}-750756824927669\, d^{8}-664023356945796\, d^{7}+3782983980383618\, d^{6}+6489582893159132\, d^{5}-24182782626411432\, d^{4}-11635999979827824\, d^{3}+98923354020446400\, d^{2}-113846941521653760\, d+40910206904985600)$

$N^o_{9} = \frac{1}{12}\, (1458\, d^{27}-13122\, d^{26}-282123\, d^{25}+2810295\, d^{24}+23620086\, d^{23}-269654670\, d^{22}-1102117023\, d^{21}+15284004291\, d^{20}+30114876816\, d^{19}-567444295476\, d^{18}-422910483264\, d^{17}+14442428462976\, d^{16}-173655449080\, d^{15}-256024966449048\, d^{14}+117007607498013\, d^{13}+3154205513887891\, d^{12}-1917888357827630\, d^{11}-26939284058835262\, d^{10}+9381328237342969\, d^{9}+170575538835999315\, d^{8}+81091482477623574\, d^{7}-1043320220595663742\, d^{6}-1048469186302651972\, d^{5}+6715930311282223672\, d^{4}+37331737163479536\, d^{3}-24266279644066088640\, d^{2}+32202247356878376960\, d-12516744443551488000)$

$N^o_{10} = \frac{1}{4}\, (1458\, d^{29}-16038\, d^{28}-324405\, d^{27}+4083129\, d^{26}+30991005\, d^{25}-471257676\, d^{24}-1603277307\, d^{23}+32578597143\, d^{22}+43377665589\, d^{21}-1500353595792\, d^{20}-174393237924\, d^{19}+48387112634196\, d^{18}-31729963605856\, d^{17}-1117426679453368\, d^{16}+1278946167008861\, d^{15}+18584197566356041\, d^{14}-25929838100941987\, d^{13}-222292527680013236\, d^{12}+301840036425933217\, d^{11}+1945082392976187963\, d^{10}-1702396861961183657\, d^{9}-13840054897129551538\, d^{8}-869477353690603586\, d^{7}+97969270160191718168\, d^{6}+43400439640826602848\, d^{5}-629779784557096225952\, d^{4}+234491922967527070944\, d^{3}+2074544035031110584960\, d^{2}-3163830549287964595200\, d+1320649590395681280000)$

$N^o_{11} = \frac{1}{4}\, (4374\, d^{31}-56862\, d^{30}-1102977\, d^{29}+16973307\, d^{28}+117608112\, d^{27}-2318176989\, d^{26}-6425985798\, d^{25}+191682452511\, d^{24}+133119473328\, d^{23}-10693744028253\, d^{22}+5846987132217\, d^{21}+424349459112114\, d^{20}-577264977532776\, d^{19}-12296681379527388\, d^{18}+23489692414637363\, d^{17}+263030758384442289\, d^{16}-593247075529299782\, d^{15}-4162564652290610993\, d^{14}+9903092880496934734\, d^{13}+49018479445283034499\, d^{12}-106385448246367215702\, d^{11}-447930561908076256091\, d^{10}+645540608889693443477\, d^{9}+3606790056461753863832\, d^{8}-1341780384161439521626\, d^{7}-28540272186090313415704\, d^{6}+1205795435057498651584\, d^{5}+182406168443172371488448\, d^{4}-128952276571759318016928\, d^{3}-557203918390573072878720\, d^{2}+976485597554969367782400\, d-435294406202292274176000)$

$N^o_{12} = \frac{1}{4}\, (13122\, d^{33}-196830\, d^{32}-3706965\, d^{31}+68070375\, d^{30}+432815319\, d^{29}-10850751657\, d^{28}-23535697932\, d^{27}+1055997538326\, d^{26}+50357440881\, d^{25}-70021319228739\, d^{24}+87374064448161\, d^{23}+3341431959709527\, d^{22}-7111742962317408\, d^{21}-118120713345379188\, d^{20}+326384557105326777\, d^{19}+3137002724874226941\, d^{18}-10106444734270420903\, d^{17}-62918353809936417707\, d^{16}+220730277300344083152\, d^{15}+956623940326083332050\, d^{14}-3393890378620526954445\, d^{13}-11241179417671261959041\, d^{12}+35363022169384877426927\, d^{11}+109549515125017919639429\, d^{10}-229128080595756761453742\, d^{9}-1002037783274877109543198\, d^{8}+840677011541967726001824\, d^{7}+8641146394045844077954112\, d^{6}-4223694280033586640137824\, d^{5}-54779602892548858064166240\, d^{4}+56217660837944500164819456\, d^{3}+156589791424366871478896640\, d^{2}-316225057234071161731737600\, d+149867365795069610096640000)$

}

\section{A node polynomial for curves in $\PP^2$}\label{Appendix B}
The number of $15$-nodal curves degree $d\geq 9$ (by \cite{KS}) in $\PP^2$ containing $\frac{d(d+1)}{2}-15$ points in general position is given by the following polynomial.

{\raggedright\doublespacing\vbadness=10000

$\frac{1}{15!} (14348907\,{d}^{30}-430467210\,{d}^{29}-789189885\,{d}^{28}+
144134770815\,{d}^{27}-800302316310\,{d}^{26}-21505566260997\,{d}^{25}
+206046709321635\,{d}^{24}+1830389081571180\,{d}^{23}-
25973085837797631\,{d}^{22}-90805122781323093\,{d}^{21}+
2106764580151475244\,{d}^{20}+1842311595032520885\,{d}^{19}-
120731061785804511795\,{d}^{18}+83105496803044790514\,{d}^{17}+
5106565375968131056197\,{d}^{16}-8800802481614659877511\,{d}^{15}-
162890506083253675564674\,{d}^{14}+397425775424906515333221\,{d}^{13}+
3952008654242554161166365\,{d}^{12}-11546375323786656779457252\,{d}^{
11}-72858625897371563437077825\,{d}^{10}+232182939704411137229570133\,
{d}^{9}+1010825449711157998476650988\,{d}^{8}-
3241105115881805786551102893\,{d}^{7}-10336040203392280930456480032\,{
d}^{6}+30163840992557581783875044832\,{d}^{5}+
74721661229894928962601063456\,{d}^{4}-168817217722446315040796818224
\,{d}^{3}-347671495806428829919633280640\,{d}^{2}+
429634898369604339129576633600\,d+794015010296634348660582144000)$

}

\bibliographystyle{mijnamsalpha}

\bibliography{bib}{}

\end{document}